\documentclass[12pt]{amsart}
\usepackage{amscd,graphicx,amsfonts,amssymb}
\def\Ac{\mathcal{A}}
\def\Bc{\mathcal{B}}
\def\Ec{\mathcal{E}}
\def\Fc{\mathcal{F}}
\def\Rc{\mathcal{R}}
\def\Sc{\mathcal{S}}

\def\Uc{\mathcal{U}}
\def\Vc{\mathcal{V}}
\def\Wc{\mathcal{W}}
\def\H{\mathbb{H}}

\def\C{\mathbb{C}}
\def\eps{\varepsilon}
\def\Im{{\rm Im}}
\def\d{\partial}

\newtheorem{theorem}{Theorem}
\newtheorem{proposition}{Proposition}
\newtheorem{lemma}{Lemma}
\newtheorem{corollary}{Corollary}

\theoremstyle{remark}

\newtheorem{example}{Example}
\newtheorem{ax}{Axiom}{\bf}{\rm}

\title{Topological regluing of rational functions}
\author{V. Timorin}

\begin{document}





\maketitle

\def\IMSmarkvadjust{0 pt}
\def\IMSmarkhadjust{0 pt}
\def\IMSmarkhpadding{0 pt}
\def\IMSpubltext{Published in modified form:}
\def\SBIMSMark#1#2#3{
 \font\SBF=cmss10 at 10 true pt
 \font\SBI=cmssi10 at 10 true pt
 \setbox0=\hbox{\SBF \hbox to \IMSmarkhpadding{\relax}
                Stony Brook IMS Preprint \##1}
 \setbox2=\hbox to \wd0{\hfil \SBI #2}
 \setbox4=\hbox to \wd0{\hfil \SBI #3}
 \setbox6=\hbox to \wd0{\hss
             \vbox{\hsize=\wd0 \parskip=0pt \baselineskip=10 true pt
                   \copy0 \break%
                   \copy2 \break%
                   \copy4 \break}}
 \dimen0=\ht6   \advance\dimen0 by \vsize \advance\dimen0 by 8 true pt
                \advance\dimen0 by -\pagetotal
	        \advance\dimen0 by \IMSmarkvadjust
 \dimen2=\hsize \advance\dimen2 by .25 true in
	        \advance\dimen2 by \IMSmarkhadjust

%
%
  \openin2=publishd.tex
  \ifeof2\setbox0=\hbox to 0pt{}
  \else 
     \setbox0=\hbox to 3.1 true in{
                \vbox to \ht6{\hsize=3 true in \parskip=0pt  \noindent  
                {\SBI \IMSpubltext}\hfil\break
                \input publishd.tex 
                \vfill}}
  \fi
  \closein2
  \ht0=0pt \dp0=0pt
 \ht6=0pt \dp6=0pt
 \setbox8=\vbox to \dimen0{\vfill \hbox to \dimen2{\copy0 \hss \copy6}}
 \ht8=0pt \dp8=0pt \wd8=0pt
 \copy8
 \message{*** Stony Brook IMS Preprint #1, #2. #3 ***}
}

\SBIMSMark{2008/4}{September 2008}{}

\begin{abstract}
Regluing is a topological operation that helps to construct topological
models for rational functions on the boundaries of certain hyperbolic components.
It also has a holomorphic interpretation, with the flavor of infinite
dimensional Thurston--Teichm\"uller theory.
We will discuss a topological theory of regluing, and trace a direction
in which a holomorphic theory can develop.
\end{abstract}

\section{Introduction}

\subsection{Overview and main results}
Consider a continuous function $f:S^2\to S^2$.
We are mostly interested in the case, where $f$ is a rational function considered as a self-map of the
Riemann sphere.
The objective is to study the topological dynamics of $f$.
In particular, how to modify the topological dynamics in a controllable way?
There is an operation that does not change the dynamics at all: a conjugation by a homeomorphism.
Let $\Phi:S^2\to S^2$ be a homeomorphism, and consider the map $g=\Phi\circ f\circ\Phi^{-1}$.
Then one can think of $g$ as being ``the map $f$ but in a different coordinate system''.
In particular, all dynamical properties of the two maps are the same.
Another example is a semi-conjugacy.
Let $\Phi:S^2\to S^2$ now be a continuous surjective map, but not a homeomorphism.
Sometimes, the ``conjugation'' $g=\Phi\circ f\circ\Phi^{-1}$ still makes sense, although
$\Phi^{-1}$ does not make sense as a map.
Namely, this happens when $f$ maps fibers of $\Phi$ to fibers of $\Phi$.
In this case, $g$ is well-defined as a continuous map.
Such map is said to be {\em semi-conjugate} to $f$.

If we want to perform a surgery on $f$, then we need to consider discontinuous maps $\Phi$
({\em if you don't cut, that is not a surgery}).
Of course, if we allow badly discontinuous maps $\Phi$, then the ``conjugation''
$g=\Phi\circ f\circ\Phi^{-1}$, even if it makes sense as a map, would have almost
nothing in common with $f$, in particular, it would be hard to say anything about
the dynamics of $g$.
Thus we must confine ourselves with only nice types of discontinuities.
An example is the following: given a simple curve, one can cut along this curve,
and then reglue in a different way.
We need to allow countably many such regluings to obtain interesting examples.
Indeed, suppose we reglued some curve.
Then we spoiled the behavior of the function near the preimages of this curve.
Thus we also need to reglue preimages, second preimages and so on.
This usually accounts to regluing of countably many curves.

An example of a regluing is the following map:
$$
j(z)=\sqrt{z^2-1}.
$$
Note that there are two branches of this map that are
well-defined and holomorphic on the complement to the segment $[-1,1]$.
We choose the branch that is asymptotic to the identity near infinity and call it $j$.
The map $j$ has a continuous extension to each ``side'' of the segment $[-1,1]$
but the limit values at different sides do not match.
By considering the limit values of $j$ at both sides of $[-1,1]$, we can say that
$j$ reglues this segment into the segment $[-i,i]$.
A precise definition of a regluing will be given in Section \ref{s_tr}.
For now, a regluing of a family of curves is a one-to-one map defined and continuous
on the complement to these curves and behaving near each curve as the map $j$ considered above.

Below, we briefly explain the main result.
Since the precise definitions are rather lengthy, we will only give a sketch,
postponing the detailed statements until the main body of the paper.
Let $\Phi$ be a regluing of countably many disjoint simple curves in the sphere
(note that the complement to countably many disjoint simple curves is not necessarily open but
is always dense).
Also, consider a ramified covering $f:S^2\to S^2$.
Under certain simple topological conditions on $f$ and the paths, we can guarantee that
the map $\Phi\circ f\circ\Phi^{-1}$ extends to the whole sphere as a ramified covering.
We say that this covering is obtained from $f$ by {\em topological regluing}.
Topological regluing is useful for constructing topological models of rational functions.

The simplest example is provided by quadratic polynomials.
E.g. consider the quadratic polynomial $f(z)=z^2-6$.
Most points drift to infinity under the iterations of $f$.
The set of points that do not is a Cantor set $J_f$ called the {\em Julia set of $f$}.
This set lies on the real line.
The right-most point of $J_f$ is 3.
Note that 3 is fixed under $f$.
The left-most point of $J_f$ is $-3$, which is mapped to 3.
The biggest component of the complement to $J_f$ in $[-3,3]$ is $(-\sqrt{3},\sqrt{3})$.
The endpoints of this interval are mapped to $-3$.
Suppose that we reglue the segment $[-\sqrt 3,\sqrt 3]$ and all its pullbacks under $f$.
Then the Julia set of $f$ collapses into a connected set.
As a result, we obtain the quadratic polynomial $g(z)=z^2-2$, the so called {\em Tchebyshev polynomial},
whose Julia set is the segment $[-2,2]$.
We will work out this example in detail in Section \ref{s_ex}.

More generally, let $f$ be a quadratic polynomial $z\mapsto z^2+c$, where $c$ is the landing
point of an external parameter ray $\Rc$.
Suppose that the Julia set of $f$ is locally connected, and all periodic points of $f$ in $\C$ are repelling.
Also, consider a quadratic polynomial $g$, for which the corresponding parameter value belongs to $\Rc$.
Thus the Julia set of $g$ is disconnected.
Then $f$ and $g$ can be obtained one from the other by a regluing.
This is explained in Section \ref{s_trqp}.
This does not say anything new about topological dynamics of quadratic polynomials but gives
a nice illustration to the notion of regluing.

However, the main motivation for the notion of regluing was the problem of finding topological
models for quadratic rational functions.
According to a well-known general observation, the dynamical behavior of a rational function
is determined by the behavior of its critical orbits.
A quadratic rational function has two critical points.
Thus, to simplify the problem, one puts restrictions on the dynamics of one critical point, and
leaves the other critical point ``free''.
For example, it makes sense to consider quadratic rational functions with one critical point
periodic of period $k$.
For $k=1$, we obtain quadratic polynomials.
Suppose now that $k>1$, and $f$ is a quadratic rational function with a $k$-periodic critical point
and the other critical point non-periodic.
Recall that $f$ is called a {\em hyperbolic rational function of type B} if
the non-periodic critical point of $f$ lies in the immediate basin of the periodic critical point
(but necessarily not in the same component, see e.g. \cite{Milnor-QuadRat,Rees_components}).
The function $f$ is said to be {\em a hyperbolic rational function of type C} if
the non-periodic critical point of $f$ lies in the full basin of the periodic critical point,
but not in the immediate basin.
This classification of hyperbolic rational functions into types was introduced by M. Rees \cite{Rees_components}.
However, a different terminology was used (types II and III instead of types B and C).
We use the terminology of Milnor \cite{Milnor-QuadRat}, which is more popular and perhaps more
suggestive (B stands for ``Bi-transitive'', and C for ``Capture'').

Fix $k>1$.
The set of hyperbolic rational functions with a $k$-periodic critical point splits into {\em hyperbolic components}.
We say that a hyperbolic component is of type B or C if it consists of hyperbolic rational functions of this type.
There are also type D components, which we will not discuss in this paper.

\begin{theorem}
\label{typeC}
  Let $f$ be a quadratic rational function with a $k$-periodic critical point.
  If $f$ is on the boundary of a type C hyperbolic component but not on the boundary of a type B
  hyperbolic component, then $f=\Phi\circ h\circ\Phi^{-1}$, where $h$ is a critically
  finite hyperbolic rational function, and $\Phi$ is a regluing of a certain countable set of paths.
  Moreover, $h$ can be chosen to be the center of a type C hyperbolic component, whose boundary contains $f$.
\end{theorem}

This result, combined with the topological models for hyperbolic critically finite functions given in
\cite{Rees_description}, provides topological models for most functions on the boundaries of type C components.
We will prove Theorem \ref{typeC} in Section \ref{s_typeC}.
The requirement that $f$ be not on the boundary of a type B component is probably inessential.
However, to study functions on the boundary of a type B component, one needs to use different techniques.
For $k=2$, there is only one type B component, and a complete description of its boundary is available \cite{Timorin}:
all functions on the boundary are simultaneously matings and anti-matings.
See also \cite{Luo,A-Y} for other interesting results concerning the case $k=2$.
On the other hand, I do not know any example of a type C hyperbolic component and a type B hyperbolic
component, whose boundaries intersect at more than one point.
In the proof of Theorem \ref{typeC}, many important ideas of \cite{Rees_description} are used.
At some point, we employ an analytic continuation argument similar to that in \cite{A-Y}.

In Section \ref{s_moore}, we prove that under some natural assumptions on a set of simple
disjoint curves on the sphere, there exists a topological regluing of this set.
This statement is a major ingredient in the proof of Theorem \ref{typeC}.
The existence result is based on a theory of Moore, which gives a topological characterization
of spaces homeomorphic to the 2-sphere.

In Section \ref{s_hr}, we define an explicit sequence of approximations to a regluing
(or even to a more general type of surgery).
These approximations are defined and holomorphic on the complements to finitely many curves.
The holomorphicity of approximations may prove to be important.
However, we just introduce the basic notions and postpone a deeper theory for future publications.

\subsection{Acknowledgements}
I am grateful to Mary Rees for hospitality and very use\-ful conversations during my visit to Liverpool.
I also had useful discussions with M. Lyubich, J. Milnor, N. Selinger.

\section{Topological regluing}
\label{s_topreg}

{\footnotesize
We define topological regluing and give simplest examples, in particular, we discuss
topological regluing of quadratic polynomials.
}

\subsection{Definition and examples}
\label{s_tr}
Let $S^1$ denote the unit circle in the plane.
In Cartesian coordinates $(t_1,t_2)$, it is given by the equation $t_1^2+t_2^2=1$.
We will also consider the sphere obtained as the one-point compactification of the
$(t_1,t_2)$-plane.
In this sphere, consider the region $\Delta_\infty$ given by the inequality
$t_1^2+t_2^2>1$ (i.e. the outside of the unit circle).
The closure of this region is denoted by $\overline\Delta_\infty$.

Let $S^2$ denote a topological sphere.
A continuous map $\alpha:S^1\to S^2$ is called an {\em $\alpha$-path} if
$\alpha(t_1,t_2)=\alpha(t'_1,t'_2)$ if and only if $t'_1=t_1$ and $t'_2=\pm t_2$
(in particular, the value of $\alpha$ at $(t_1,t_2)$ depends only on $t_1$).
A continuous map $\beta:S^1\to S^2$ is called a {\em $\beta$-path} if
$\beta(t_1,t_2)=\beta(t'_1,t'_2)$ if and only if $t'_1=\pm t_1$ and $t'_2=t_2$.
Note that every simple path $\gamma:[-1,1]\to S^2$ can be interpreted either as an
$\alpha$-path defined as $\alpha(t_1,t_2)=\gamma(t_1)$ or as a $\beta$-path defined as
$\beta(t_1,t_2)=\gamma(t_2)$.
It is easy to see that every $\alpha$-path $\alpha:S^1\to S^2$ can be extended to
a continuous map $\hat\alpha:\overline\Delta_\infty\to S^2$ such that $\hat\alpha$
is a homeomorphism between $\Delta_\infty$ and $S^2-\alpha(S^1)$.
The extension $\hat\alpha$ is defined uniquely up to a homotopy relative to $S^1$, if
we require additionally that it preserve orientation.
Similarly, every $\beta$-path admits a continuous orientation preserving
extension $\hat\beta:\overline\Delta_\infty\to S^2$
that is a homeomorphism between $\Delta_\infty$ and $S^2-\beta(S^1)$ and that is defined uniquely
up to a homotopy relative to $S^1$.

We can now give a definition of a regluing.
Let $\Ac$ be a set of disjoint $\alpha$-paths in the sphere
(being disjoint means that $\alpha(S^1)\cap\alpha'(S^1)=\emptyset$
for every pair of different $\alpha$-paths $\alpha,\alpha'\in\Ac$).
Define the set $\Im\Ac\subset S^2$ as the union of $\alpha(S^1)$ for all $\alpha\in\Ac$.
Consider also a set $\Bc$ of disjoint $\beta$-paths in the sphere.
The set $\Im\Bc\subset S^2$ is defined similarly.
A bijective continuous map $\Phi:S^2-\Im\Ac\to S^2-\Im\Bc$ is called a regluing of $\Ac$ into $\Bc$
according to a given one-to-one correspondence between $\Ac$ and $\Bc$ if
for every $\alpha\in\Ac$ and the corresponding $\beta\in\Bc$, the map $\Phi\circ\hat\alpha$
extends to $\overline\Delta_\infty$ in such a way that this extension is continuous on $S^1$ and
coincides with $\beta$ on $S^1$
(note, however, that in general the map $\Phi\circ\hat\alpha$ will not be continuous on $\overline\Delta_\infty$,
even in arbitrarily small neighborhood of $S^1$).

\begin{figure}
\centering
\includegraphics[width=11cm]{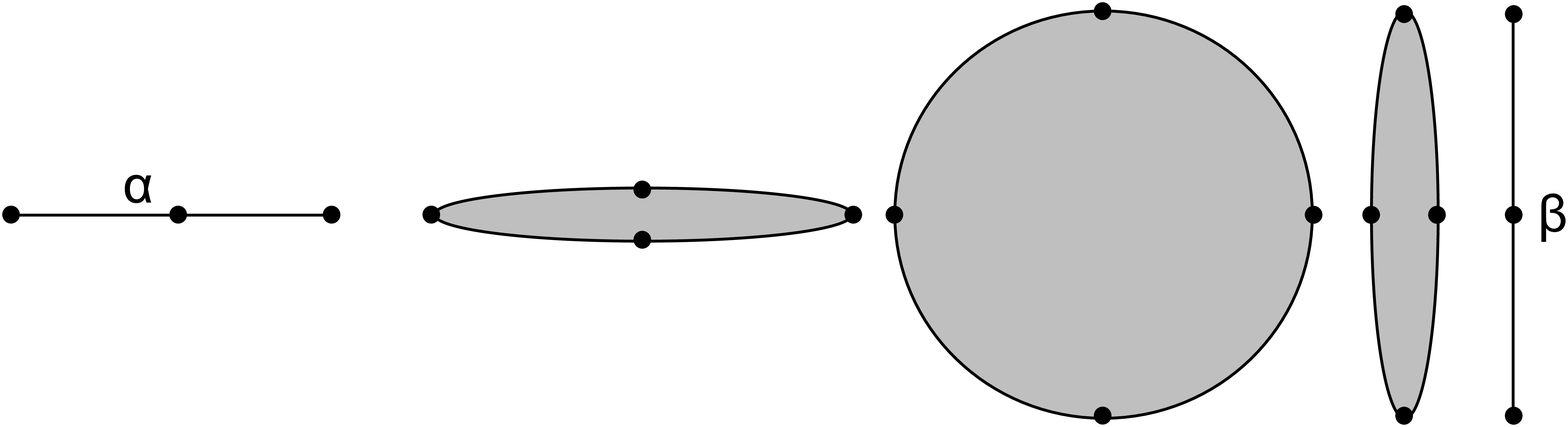}
\caption{A schematic picture of regluing}
\end{figure}

\begin{example}
  As an example, consider the following function
  $$
  j(z)=\sqrt{z^2-1}
  $$
  on the Riemann sphere with the segment $[-1,1]$ removed.
  This function maps $\overline\C-[-1,1]$ to $\overline\C-[-i,i]$ homeomorphically.
  Set $\alpha(t_1,t_2)=t_1$ and $\beta(t_1,t_2)=it_2$.
  Then $j$ reglues $\alpha$ into $\beta$.
\end{example}

\begin{proposition}
  \label{reg_homeo}
  Consider a regluing $\Phi$ of a set $\Ac$ of disjoint $\alpha$-paths into a set $\Bc$ of disjoint $\beta$-paths.
  The map $\Phi:S^2-\Im\Ac\to S^2-\Im\Bc$ is a homeomorphism.
\end{proposition}

\begin{proof}
  We know by definition that this map is bijective and continuous.
  Thus it remains to prove that $\Phi^{-1}:S^2-\Im\Bc\to S^2-\Im\Ac$ is continuous.
  Consider a sequence $y_n\in S^2-\Im\Bc$ that converges to a point $y$ in $S^2-\Im\Bc$, and
  an accumulation point $x$ of the sequence $x_n=\Phi^{-1}(y_n)$ in $S^2$.
  If $x\not\in\Im\Ac$, then we must have $y=\Phi(x)$ by continuity of $\Phi$.
  Such point $x$ is unique.
  If $x=\alpha(t_1,t_2)$ for some $(t_1,t_2)\in S^1$, then from the definition of regluing it
  follows that $y=\beta(t_1,\pm t_2)$, a contradiction.
  Thus there is only one accumulation point, and the sequence $x_n$ converges to $x$.

\end{proof}

For every $\alpha$-path $\alpha:S^1\to S^2$, define a $\beta$-path $\alpha^{\#}$ as follows:
$\alpha^{\#}(t_1,t_2)=\alpha(t_2,t_1)$.
Similarly, the formula $\beta^{\#}(t_1,t_2)=\beta(t_2,t_1)$ makes an $\alpha$-path out of a $\beta$-path $\beta$.
For a set $\Ac$ of $\alpha$- or $\beta$-paths, we can form the set $\Ac^{\#}=\{\alpha^{\#}\ |\ \alpha\in\Ac\}$.

\begin{proposition}
\label{inv}
  Let $\Phi$ be a regluing of a set $\Ac$ of disjoint $\alpha$-paths into a set $\Bc$
  of disjoint $\beta$-paths. Then $\Phi^{-1}$ is a regluing of $\Bc^{\#}$ into $\Ac^{\#}$.
\end{proposition}

\begin{proof}
  By Proposition \ref{reg_homeo}, the map $\Phi^{-1}:S^2-\Im\Ac\to S^2-\Im\Bc$ is a homeomorphism.
  Note that $\Im\Ac^{\#}=\Im\Ac$ and $\Im\Bc^{\#}=\Im\Bc$.
  For every $\alpha\in\Ac$, the composition $\Phi\circ\hat\alpha$ is asymptotic to $\hat\beta$
  near the unit circle, where $\beta$ is the $\beta$-path corresponding to the $\alpha$-path $\alpha$.
  It follows that $\Phi^{-1}\circ\hat\beta$ is asymptotic to $\hat\alpha$ on the unit circle.
\end{proof}

Let $f:S^2\to S^2$ be a continuous map.
Assume that a countable set $\Ac$ of disjoint $\alpha$-paths satisfies the following conditions:
\begin{itemize}
  \item {\em Forward semi-invariance:} for any path $\alpha\in\Ac$, we have $f\circ\alpha\in\Ac$
  or $f\circ\alpha(t_1,t_2)=f\circ\alpha(-t_1,t_2)$ for all $(t_1,t_2)\in S^1$.
  In the latter case, $\alpha(S^1)$ must be disjoint from $\Im\Ac$.
  \item {\em Backward invariance:} we have $f^{-1}(\Im\Ac)\subseteq\Im\Ac$.
\end{itemize}
We say in this case that $\Ac$ is {\em $f$-stable}.
Our main construction is based on the following simple fact:

\begin{theorem}
\label{contin}
  Suppose that $f:S^2\to S^2$ is a continuous map, and $\Ac$ is an $f$-stable set of disjoint $\alpha$-paths.
  Let $\Phi$ be a regluing of $\Ac$ into a set $\Bc$ of disjoint $\beta$-paths.
  Then the map $g=\Phi\circ f\circ\Phi^{-1}$ extends to a continuous map from $S^2$ to $S^2$.
\end{theorem}

\begin{proof}
  Note that the set $S^2-\Im\Ac$ is forward invariant under $f$, which
  follows from the backward invariance of $\Ac$.
  In particular, the map $g$ is well-defined and continuous on $S^2-\Im\Bc$.
  Consider a sequence $y_n\in S^2-\Im\Bc$ converging to $\beta(t_1,t_2)$ for some $\beta\in\Bc$ and $(t_1,t_2)\in S^1$.
  Then the sequence $x_n=\Phi^{-1}(y_n)$ can have at most two accumulation points,
  namely, $\alpha(\pm t_1,t_2)$, where $\alpha\in\Ac$ is the $\alpha$-path corresponding
  to the $\beta$-path $\beta$.
  By the forward semi-invariance, we have two cases: either $f\circ\alpha\in\Ac$, or the two values $f\circ\alpha(\pm t_1,t_2)$
  coincide and lie in the complement to $\Im\Ac$.
  In the first case, we have $f\circ\hat\alpha=\hat\alpha_1$ for some $\alpha_1\in\Ac$ up to a suitable homotopy,
  and $\Phi\circ\hat\alpha_1$ extends to $\overline\Delta_\infty$ in such a way
  that the extension is continuous on $S^1$ and coincides with $\beta_1\in\Bc$ on $S^1$.
  It follows that $\Phi\circ f\circ\Phi^{-1}$ extends continuously into $\beta(S^1)$ and sends $\beta(t_1,t_2)$ to $\beta_1(t_1,t_2)$.
  In the second case, the image $\Phi\circ f\circ\alpha(\pm t_1,t_2)$ is well-defined and unique.
\end{proof}

We would like to apply this theorem as follows.
Let $f:\overline\C\to\overline\C$ be a rational function.
For certain classes of rational functions $f$, there are natural ways to
produce $f$-stable sets of paths.
Then the corresponding map $g$ is supposed to give a model for a new rational function.
Note that the topological dynamics of $g$ is very easy to understand in terms of the topological
dynamics of $f$, because $\Phi$ is a topological conjugation except on $\Im(\Ac)$.
A remarkable fact is that in many cases, the regluing $\Phi$ makes sense in
a certain holomorphic category, so that the construction may actually produce a rational
function $g$ rather than just a continuous map.

Let $X$ be a compact metric space, and $\Ac$ a set of compact subsets of $X$.
We say that $\Ac$ is {\em contracted} if for every $\eps>0$, there are only finitely
many elements of $\Ac$, whose diameter exceeds $\eps$.
The following theorem is needed for the construction of topological models:

\begin{theorem}
\label{conj_tr}
  Let $\Ac$ be a countable contracted set of disjoint $\alpha$-paths.
  Then there exists a regluing of $\Ac$ into some set $\Bc$ of disjoint $\beta$-paths.
  Moreover, one can arrange $\Bc$ to be contracted.
\end{theorem}

In the statement of the theorem, it is said that the set $\Ac$ of paths should be contracted.
This is a minor abuse of terminology.
To be more precise, the set of connected components of $\Im\Ac$ is assumed to be contracted.
The statement of the theorem may seem intuitively obvious (and it is in fact obvious for the
case of finite $\Ac$).
Note, however, that the set $\Im\Ac$ may be everywhere dense in the sphere, and
even have full measure.
We will prove Theorem \ref{conj_tr} in Section \ref{s_regs} using Moore's theory.
It is useful to know that the property of being contracted is topological, and does not
depend on a particular metric:

\begin{proposition}
\label{contracted}
  Let $X$ be a compact metric space.
  A set $\Ac$ of compact subsets of $X$ is contracted if and only if for every
  open covering $\Ec$ of $X$, there is a finite subset $\Ac'\subset\Ac$ such that
  every element of $\Ac-\Ac'$ is contained in an element of $\Ec$.
\end{proposition}

In other terms, the number $\eps>0$ can be replaced with an open covering $\Ec$.

\begin{proof}
  First prove the only if part.
  Let $\Ec$ be an open cover of $X$, and $\eps>0$ its Lebesgue number.
  Recall that a Lebesgue number of $\Ec$ is defined as a positive real number $\eps>0$ such that
  every set of points of diameter less than $\eps$ belong to a single element of $\Ec$.
  Set $\Ac'$ to be the set of all elements of $\Ac$, whose diameter is at least $\eps$.
  Then, by the Lebesgue number lemma, every element of $\Ac-\Ac'$ is contained in
  an element of the cover $\Ec$.

  Let us now prove the only if part.
  Choose any $\eps>0$, and consider the covering $\Ec$ of $X$ by all $\eps$-balls.
  Then there is a finite subset $\Ac'\subset\Ac$ such that every set in $\Ac-\Ac'$
  is contained in an element of $\Ec$.
  It follows that the diameter of any set in $\Ac-\Ac'$ does not exceed $\eps$.
\end{proof}

\begin{corollary}
\label{contracted_image}
  Let $X$ and $Y$ be compact metric spaces, and $\phi:X\to Y$ a continuous map.
  If $\Ac$ is a contracted family of compacts sets in $X$, then
  $\phi(\Ac)=\{\phi(A)\ |\ A\in\Ac\}$ is a contracted family of compact sets in $Y$.
\end{corollary}

\begin{proof}
  Indeed, let $\Ec$ be any open cover of $Y$.
  Consider the corresponding cover $\phi^*(\Ec)$ of $X$.
  By Proposition \ref{contracted}, all elements of $\Ac$ but finitely many are subsets of
  elements of $\phi^*(\Ec)$.
  It follows that all elements of $\phi(\Ac)$ but finitely many are subsets of elements of $\Ec$.
\end{proof}

\subsection{Regluing of ramified coverings}
The setting of ramified coverings is most commonly used for topological discussions  of rational functions.
On one hand, ramified coverings are objects of topological nature, and are much
more flexible than holomorphic functions.
On the other hand, they are nice objects and do not have pathologies of general continuous maps.
This is why we want the regluing construction to fit into the contest of
topological ramified coverings.

Let $f:S^2\to S^2$ be a ramified covering.
Consider a contracted set $\Ac$ of simple disjoint $\alpha$-paths in the sphere.
We say that $\Ac$ is {\em strongly $f$-stable} if it is $f$-stable, and satisfies the following additional
assumption: all critical points in $\Im\Ac$ have the form $\alpha(0,1)$, where
$\alpha\in\Ac$ is a path such that $f\circ\alpha(t_1,t_2)=f\circ\alpha(-t_1,t_2)$
for all $(t_1,t_2)\in S^1$; moreover, these critical points are simple.

\begin{theorem}
  \label{reg_rc}
  Let $f:S^2\to S^2$ be a topological ramified covering, and $\Ac$ a contracted strongly $f$-stable set of
  of disjoint $\alpha$-paths.
  Consider a regluing $\Phi:S^2-\Im\Ac\to S^2-\Im\Bc$ of $\Ac$ into some contracted set $\Bc$
  of disjoint $\beta$-paths.
  Then the map $g=\Phi\circ f\circ\Phi^{-1}$ extends by continuity to a ramified self-covering of $S^2$.
\end{theorem}

We will prove this theorem in Section \ref{s_ramcov}.
In the statement of the theorem, we assumed that $\Bc$ is contracted.
We do not need to check this, however, because this can always be arranged by Theorem \ref{conj_tr}.
On the other hand, the condition is superfluous, because for any regluing $\Phi$ of a contractible
set $\Ac$ of disjoint $\alpha$-paths into some set $\Bc$ of disjoint $\beta$-paths, the set $\Bc$
will automatically be contracted.
Indeed, by Theorem \ref{conj_tr}, there is a regluing $\Phi'$ of $\Ac$ into some contracted set $\Bc'$
of disjoint $\beta$-paths.
Then $\Phi\circ(\Phi')^{-1}$ extends to the sphere as a continuous map.
Moreover, it maps $\Bc'$ to $\Bc$.
By Corollary \ref{contracted_image}, the set $\Bc$ must also be contracted.

\subsection{Topological regluing of quadratic polynomials}
\label{s_trqp}
In this section, we will not say anything new about the dynamics of quadratic polynomials.
However, we can illustrate the idea of regluing using quadratic polynomials as an example.

Let $\Rc$ be an external ray in the parameter plane of quadratic polynomials
(we write quadratic polynomials in the form $p_c(z)=z^2+c$, thus the parameter plane is the $c$-plane).
Suppose that $\Rc$ lands at a point $c$ on the boundary of the Mandelbrot set.
Suppose that the Julia set of $f=p_c$ is locally connected, and that all periodic points of $f$ (except $\infty$)
are repelling.
The ray $\Rc$ determines a pair of rays $R^+_f$ and $R^-_f$ in the dynamical plane of $f$
that land at the critical point $0$ (for parameter values in the ray $\Rc$, these
two rays crash into $0$).
Note that there may be more rays landing at $0$, but the pair of rays $R^+_f$, $R^-_f$ is distinguished.

Fix any real number $\rho>0$.
Consider the $\alpha$-path $\alpha_0:S^1\to\overline\C$ in the dynamical plane of $f$
defined as follows:
$$
\alpha_0(t_1,t_2)=\left\{\begin{array}{cl}
R^+_f(t_1\rho),& t_1>0,\\
0,& t_1=0,\\
R^-_f(-t_1\rho),& t_1<0.
\end{array}\right.
$$
Here the dynamical rays are parameterized by the values of the Green function,
thus $R(t)$ stands for the point in the ray $R$, at which the Green function is equal to $t$.
Then, for each $n\ge 0$, the multivalued function $f^{-n}\circ\alpha_0$ has $2^n$ branches,
each being an $\alpha$-path.
All these paths are called {\em pullbacks} of $\alpha_0$ under the iterates of $f$.
Let $\Ac$ denote the set of such pullbacks, including $\alpha_0$.
Clearly, $\Ac$ is strongly $f$-stable.

Now consider the quadratic polynomial $g=p_{c_0}$, where $c_0=\Rc(2\rho)$ is the
point on the external parameter ray $\Rc$ with parameter $2\rho$ (the
external parameter rays are parameterized by the value of the Green function at the
critical value).
This means that, in the dynamical plane of $g$, the value of the Green function
at the critical value $c_0$ is equal to $2\rho$.
Therefore, the value of the Green function at the critical point $0$ is equal to $\rho$.
There are exactly two rays that are bounded and contain $0$ in their closures (here by a ray we mean a gradient curve
of the Green function).
Denote these rays by $R^+_g$ and $R^-_g$.
These two rays can also be parameterized by the values of the Green function,
thus the parameter runs through the interval $(0,\rho)$.
Let $z^+$ and $z^-$ be the landing points of the rays $R^+_g$ and $R^-_g$
(these rays land because the angle of $\Rc$ cannot be a rational number with an odd denominator,
which is the only case when one of the rays $R^+_g$ and $R^-_g$ can crash into a precritical point).
Define the following $\beta$-path $\beta_0:S^1\to\overline\C$ in the dynamical plane of $g$:
$$
\beta_0(t_1,t_2)=\left\{\begin{array}{cl}
z^+,& t_2=1,\\
R^+_g(|t_1|\rho),& t_2>0,\\
0,& t_2=0,\\
R^-_g(|t_1|\rho),& t_2<0,\\
z^-,& t_2=-1.
\end{array}\right.
$$
Let $\Bc$ denote the set of all pullbacks of $\beta_0$, including $\beta_0$.

There is a natural one-to-one correspondence between the sets of paths $\Ac$ and $\Bc$.
For any path $\alpha\in\Ac$, the point $\alpha(1,0)$ belongs to a unique external ray of angle $\theta$.
There is a unique path $\beta\in\Bc$ such that the ray of angle $\theta$ crashes into $\beta(1,0)$.
We will make this path $\beta$ correspond to the path $\alpha$.

\begin{theorem}
\label{troqp}
  There exists a regluing $\Phi$ of $\Ac$ into $\Bc$ such that $g(x)=\Phi\circ f\circ\Phi^{-1}(x)$
  at all points $x$, where the right-hand side is defined.
\end{theorem}

Before proceeding with the proof of this theorem, we need to recall the definition of
kneading sequences.
  The union $\Gamma_f$ of the rays $R_f^+$ and $R_f^-$ together with their common landing point $0$
  divides the complex plane into two parts: the positive part and the negative part.
  We label the two parts positive or negative arbitrarily, the only requirement is that one
  part be positive and one part be negative.
  For a point $x$ not in $f^{-n}(\Gamma_f)$, we define $\sigma^n_f(x)$ to be $1$ or $-1$
  depending on whether $f^n(x)$ belongs to the positive or to the negative part of the plane.
  The sequence of numbers $\sigma^n_f(x)$ (which may be finite or infinite depending on whether
  or not $x$ is eventually mapped to $\Gamma_f$) is called the {\em kneading sequence} of $x$.
  Similarly, we define $\Gamma_g$ to be the union of $\{0\}$ and the external rays in the dynamical plane of $g$
  that crash into $0$ (they have the same external angles as the rays $R^+_f$ and $R^-_f$).
  The definition of kneading sequences carries over to the dynamical plane of $g$, where we use $\Gamma_g$ instead of $\Gamma_f$.
  However, the positive and negative parts in the dynamical plane of $f$ should correspond
  to the positive and negative parts in the dynamical plane of $g$, i.e. the corresponding parts
  should contain rays of the same angles.
  In the dynamical plane of $g$, as well as in the dynamical plane of $f$, there cannot be
  two different points in the Julia set with the same kneading sequence.
  This is a basic Poincar\'e distance argument, see e.g. \cite{Milnor}.

\begin{proof}[Proof of Theorem \ref{troqp}]
  Consider the complement $U$ to the closure of $\Im\Ac$.
  Since the closure of $\Im\Ac$ contains the Julia set of $f$, the set $U$ is an
  open subset of the Fatou set.
  Actually, $U$ is the complement in the Fatou set to $\Im\Ac$.

  We first define the map $\Phi$ just on $U$.
  Let $\phi_f$ be the B\"ottcher parameterization for $f$, i.e. the holomorphic automorphism
  between the unit disk $\Delta$ and the Fatou set of $f$ such that $\phi_f(z^2)=f\circ\phi_f(z)$ for all $z\in\Delta$.
  Similarly, we define $\phi_g$ to be the B\"ottcher parameterization for $g$.
  Set $\Phi=\phi_g\circ\phi_f^{-1}$.
  Clearly, we have $g=\Phi\circ f\circ\Phi^{-1}$ on $\Phi(U)$.
  Note also that $\Phi$ preserves the values of the Green function.

  It is easy to see that $\Phi$ extends continuously to each side of each path $\alpha\in\Ac$.
  The extension preserves the values of the Green function.
  It follows that, for every $\alpha\in\Ac$, the function $\Phi\circ\hat\alpha$ extends to the unit
  circle as the corresponding $\beta$-path $\beta\in\Bc$, and the extension is continuous on the
  unit circle.

  We now need to show that for any point $z\in\C-\Im\Ac$ and any sequence $z_n\in U$ converging to $z$,
  the sequence $w_n=\Phi(z_n)$ converges to a well-defined point in the dynamical plane of $g$,
  and this point does not depend on the choice of the sequence $z_n$.
  Indeed, any limit point of the sequence $w_n$ must have the same kneading sequence as $z$,
  therefore, this can only be one point.
  We denote this point by $\Phi(z)$, which is justified by the fact that $\Phi$ extends continuously to $z$.
  Note that, for different $z$, the points $\Phi(z)$ have different kneading sequences, and hence are different.
  This finishes the proof of the theorem.
\end{proof}

The reason for the proof shown above to be so simple is that we know a lot about
topological dynamics of both $f$ and $g$.
However, we would like to use regluing to describe new kinds of topological dynamics, and
that would be necessarily more complicated.

\subsection{Topological models via regluing}
\label{s_tmvr}
First, we need to make the notion of topological model more precise.
Define an (abstract) {\em topological model} as the collection of the following data:
\begin{itemize}
\item
A ramified topological covering $f:X\to X$, where $X$ is a topological space
homeomorphic to the 2-sphere.
\item
A compact fully invariant subset $J\subset X$, called the {\em Julia set} of $f$.
The complement to the Julia set is called the {\em Fatou set}.
\item
A complex structure (i.e. a Riemann surface structure) on the Fatou set
such that $f$ is holomorphic with respect to this structure.
\end{itemize}
The topological space $X$, the map $f$ and the set $J$ are called the
{\em model space}, the {\em model map} and the {\em model Julia set}, respectively.
Instead of referring to a topological model as $(f,X,J)$, we will sometimes
simply say ``topological model $f$''.
We will sometimes call $X$ the {\em dynamical sphere} of $f$.

Although we do not require anything else for the definition, it is usually
meant that a topological model should have a simple explicit dynamical behavior.
This is the reason why we do not require an invariant complex structure to be defined
on the whole space $X$: in most cases, it is hard to do this explicitly.
On the other hand, it is relatively easy to define an explicit invariant complex structure
just on the Fatou set.

Of course, any rational function is an abstract topological model.
We say that two abstract topological models $(f,X_f,J_f)$ and $(g,X_g,J_g)$ are {\em equivalent} if
there is a homeomorphism $\phi:X_f\to X_g$ conjugating $f$ with $g$ and
such that $\phi(X_f-J_f)=X_g-J_g$ and $\phi|_{X_f-J_f}$ is holomorphic.
We say that a topological model $f$ {\em models} a rational function $R$
if $f$ is equivalent to $R$ as an abstract topological model.

There are several important combinatorial constructions that modify or combine
topological models into new topological models.
Among the most well-known are matings and captures.
Let us now define another combinatorial operation on topological models that uses regluing.
There are interesting relationships between matings, captures and regluings,
which we may discuss elsewhere.

Consider a topological model $(f,X,J)$, and
a strongly $f$-stable set $\Ac$ of disjoint $\alpha$-paths in $X$.
Define an {\em accumulation point of $\Ac$} as a point $x\in X$ such that every open
neighborhood of $x$ intersects infinitely many elements of $\Ac$.
We will make the following assumptions on $\Ac$:
\begin{enumerate}
  \item\label{ass1} the set $\Ac$ is contracted;
  \item\label{ass2} all accumulation points of $\Ac$ belong to the Julia set of $f$;
  \item\label{ass3} for every $\alpha\in\Ac$, there exists $n>0$ such that $f^{\circ n}\circ\alpha\not\in\Im\Ac$.
\end{enumerate}
Under these assumptions, we will define another topological model using regluing.

By assumption \ref{ass1}, there exists a regluing $\Phi:X-\Im\Ac\to Y-\Im\Bc$ of $\Ac$ into some other contracted set
$\Bc$ of simple paths in a topological sphere $Y$.
We set $g$ to be the continuous extension of $\Phi\circ f\circ\Phi^{-1}$, which exists by Theorem \ref{contin}.
Since $f$ is a ramified covering, and $\Ac$ is contracted and strongly $f$-stable, the map $g$
is also a ramified covering by Theorem \ref{reg_rc}.
Define the Fatou set of $g$ as the set of all points $y\in Y$ such that,
for some nonnegative integer $n$, we have $g^{\circ n}(y)\not\in\Im\Bc$ (thus $\Phi^{-1}$
is defined at this point) and $\Phi^{-1}(g^{\circ n}(y))\in X-J$.
Clearly, if this condition is satisfied for one particular $n$, then it also holds for all bigger $n$.
Therefore, the Fatou set of $g$ thus defined is fully invariant.

It remains to define a complex structure on the Fatou set of $g$ invariant under $g$.
Take a point $y$ in the Fatou set of $g$ and its iterated image $y'=g^{\circ n}(y)$ such
that $y'\not\in\Im\Bc$ and $x'=\Phi^{-1}(y')$ is in the Fatou set of $f$.
Since $\Ac$ does not accumulate in the Fatou set of $f$, there is a neighborhood $U$ of $x'$
disjoint from $\Im\Ac$ and a holomorphic embedding $\xi:U\to\C$ such that $\xi(x')=0$.
Note that $\Phi:U\to\Phi(U)$ is a homeomorphism.
Therefore, $\xi\circ\Phi^{-1}$ is an embedding of the neighborhood $\Phi(U)$ of $y'$ into $\C$.
Finally, if $k$ is the local degree of $g^{\circ n}$ at $y$, then we can define
a local complex coordinate near $y$ as a branch of $\sqrt[k]{\xi\circ\Phi^{-1}\circ g^{\circ n}}$.
We have now defined a complex coordinate near every point of the Fatou set of $g$.
It is easy to check that all transition functions are holomorphic, and that $g$ is holomorphic
with respect to the obtained complex structure on the Fatou set.

We defined a topological model $g$.
This topological model will be called the topological model obtained from $(f,X,J)$ by regluing of $\Ac$.

\section{Boundary points of type C hyperbolic components}
\label{s_typeC}
{\footnotesize
In this section, we discuss some general properties of quadratic rational functions on the
boundaries of type C hyperbolic components, preparing for the proof of Theorem \ref{typeC}.
}

\subsection{Equicontinuous families and holomorphic motion}
Below we recall a standard fact about equicontinuous families widely used in
complex dynamics:

\begin{proposition}
\label{ec}
  Let $\Lambda$ be a topological space, and $X$ a metric space.
  Consider an equicontinuous family $\Fc$ of maps $f:\Lambda\to X$.
  Let $O\subset\Lambda$ be an open subset and $\nu:\Lambda\to X$ a
  continuous map such that $\nu(\lambda)\in\Fc(\lambda)$ for all $\lambda\in O$.
  Then for all $\lambda^*\in\d O$, we have
  $$
  \nu(\lambda^*)\in\overline{\Fc(\lambda^*)}.
  $$
\end{proposition}

\begin{proof}
  Assume the contrary: $d(\nu(\lambda^*),\Fc(\lambda^*))=\eps>0$,
  where $d$ denotes the distance in $X$, and the distance between a point and a
  set is defined as the infimum of distances between this point and points in the set.
  There is a neighborhood $V'$ of $\lambda^*$ in $\Lambda$ such that
  $$
  d(\nu(\lambda),\nu(\lambda^*))<\frac\eps 2
  $$
  for all $\lambda\in V'$.
  This follows from the continuity of $\nu$.
  On the other hand, there is a neighborhood $V''$ of $\lambda^*$ such that
  $$
  d(f(\lambda),f(\lambda^*))<\frac\eps 2
  $$
  for all $\lambda\in V''$ and all $f\in\Fc$.
  This follows from the equicontinuity of $\Fc$.
  Therefore, for every $\lambda\in V'\cap V''$ and every $f\in\Fc$, we have
  $$
  d(\nu(\lambda^*),f(\lambda^*))\le d(\nu(\lambda^*),\nu(\lambda))+
  d(\nu(\lambda),f(\lambda))+d(f(\lambda),f(\lambda^*))<
  $$
  $$
  <d(\nu(\lambda),f(\lambda))+\eps.
  $$
  Take $\lambda\in V'\cap V''\cap O$ and $f\in\Fc$ such that $\nu(\lambda)=f(\lambda)$.
  Then $d(\nu(\lambda^*),f(\lambda^*))<\eps$, a contradiction.
\end{proof}

Let $\Lambda$ be a complex analytic manifold, and $A$ a set.
Recall that a {\em holomorphic motion} over $\Lambda$ is a map $\mu:\Lambda\times A\to\overline\C$
such that $\mu(\lambda,a)\ne\mu(\lambda,b)$ for $a\ne b$ and the map
$\mu_a:\lambda\mapsto\mu(\lambda,a)$ is holomorphic for every $a\in A$.
We do not require that $A\subset\overline\C$ and that $a\mapsto\mu(\lambda_0,a)$
is the identity for some $\lambda_0\in\Lambda$, although these conditions are
usually included into a definition.
Thus we use the term ``holomorphic motion'' in a slightly more general sense.
The following well-known fact is very simple but important (see e.g. \cite{MSS}):

\begin{theorem}
\label{mss_eq}
  Let $\Lambda$ be a Riemann surface and $\mu:\Lambda\times A\to\overline\C$ a holomorphic motion.
  Then the family of functions $\mu_a$, $a\in A$, is equicontinuous.
\end{theorem}

\begin{proof}
  If $A$ is finite, then the statement is obvious.
  Suppose that $A$ is infinite, and take three different points $a_1,a_2,a_3\in A$.
  We can use the following generalization of Montel's theorem:
  if a family of holomorphic functions is such that the graphs of all functions
  in the family avoid the graphs of three different holomorphic functions, and these three graphs
  are disjoint, then the family is equicontinuous.
  In our case, we can take $\mu_{a_i}$, $i=1,2,3$.
  These three holomorphic functions have disjoint graphs, and the graph of any
  function $\mu_a$, $a\ne a_1$, $a_2$, $a_3$, is disjoint from the graphs of $\mu_{a_i}$.
  Thus the family of functions $\mu_a$ is equicontinuous.
\end{proof}

The following well-known theorem is proved in \cite{MSS}:

\begin{theorem}
\label{mss_ext}
  Suppose now that $A\subset\overline\C$ and that $\mu(\lambda_0,a)=a$
  for some $\lambda_0\in\Lambda$ and all $a\in A$.
  Then $\mu$ extends to a holomorphic motion
  $\overline\mu:\Lambda\times\overline A\to\overline\C$, and, for every $\lambda\in\Lambda$,
  the map $a\mapsto\overline\mu(\lambda,a)$ from $\overline A$ to $\overline\C$
  is quasi-symmetric.
\end{theorem}

Using this theorem, we can prove the following (cf. e.g. \cite{A-Y}):

\begin{proposition}
\label{dA}
  Consider a holomorphic motion $\mu$ satisfying the assumptions of Theorem \ref{mss_ext},
  with $\Lambda$ being a Riemann surface.
  Assume that $A$ is an open subset of $\overline\C$.
  Consider a continuous function $\nu:\Lambda\to\overline\C$ and the subset $O$ of $\Lambda$
  consisting of all $\lambda\in\Lambda$ such that $\nu(\lambda)\in\mu(\lambda,A)$.
  Then $O$ is open.
  Moreover, $\nu(\lambda^*)\in\d\mu(\lambda,A)$ if $\lambda^*\in\d O$.
\end{proposition}

\begin{proof}
  Consider a point $\lambda_0\in O$.
  Then $\nu(\lambda_0)=\mu(\lambda_0,a_0)$ for some $a_0\in A$.
  Assume that $\nu(\lambda_0)\ne\infty$.
  Let $\alpha:S^1\to A$ be a small enough loop around $a_0$.
  In particular, $\mu_{\lambda_0}$ is a holomorphic function on the disk bounded by $\alpha(S^1)$.
  Then we have
  $$
  I(\lambda_0)=\int_{S^1}\frac{d\mu(\lambda_0,\alpha(t))}{\mu(\lambda_0,\alpha(t))-\nu(\lambda_0)}=2\pi i.
  $$
  Set $\eps$ to be the minimal spherical distance between $\nu(\lambda_0)$ and $\mu(\lambda_0,\alpha(t))$.
  There is an open neighborhood $V$ of $\lambda_0$ such that the functions $\mu_\lambda$, $\lambda\in V$, are holomorphic on
  the disk bounded by $\alpha(S^1)$, and the
  distance between $\nu(\lambda)$ and $\mu(\lambda,\alpha(t))$ is bigger than $\eps/2$ for
  all $\lambda\in V$ and all $t\in S^1$.
  This follows from the continuity of $\nu$ and equicontinuity of $\mu(\cdot,a)$, $a\in A$.
  Then the integral
  $$
  I(\lambda)=\int\frac{d\mu(\lambda,\alpha(t))}{\mu(\lambda,\alpha(t))-\nu(\lambda)}
  $$
  is well-defined and continuous function on $V$.
  Since the possible values of this integral are discrete, we must have $I(\lambda)=2\pi i$
  for all $\lambda\in V$.
  Therefore, $\nu(\lambda)\in\mu(\lambda,A)$ for such $\lambda$, and $V\subset O$.
  Thus we proved that $O$ is open.

  Suppose now that $\lambda^*\in\d O$.
  Then $\nu(\lambda^*)\in\overline{\mu(\lambda^*,A)}$ by Proposition \ref{ec}.
  On the other hand, we have $\nu(\lambda^*)\not\in\mu(\lambda^*,A)$ because $\lambda^*\not\in O$.
  Therefore, $\nu(\lambda^*)\in\d\mu(\lambda,A)$.
\end{proof}

\subsection{Parameter curves}
Quadratic rational functions that are conjugate by a M\"obius transformation have the same dynamical properties.
Therefore, one wants to parameterize conjugacy classes, choosing one (or finitely many) particular representative(s)
from every conjugacy class.
There are many different ways to do this parameterization, see e.g. \cite{Milnor-QuadRat,ReesV3}.
For our purposes, it will be convenient to do the following: send the two critical points of
a rational function to $0$ and $\infty$ by a suitable M\"obius transformation.
If $\infty$ is fixed, then we can reduce $f$ to the form $p_c:z\mapsto z^2+c$.
If $\infty$ is not fixed, then we can send a preimage of $\infty$ to 1, thus $f$ will have the form
$$
R_{a,b}(z)=\frac{az^2-b}{z^2-1}.
$$
In any case, we can assume that $f$ is either $p_c$ or $R_{a,b}$.

We will now consider the following algebraic curves in $\C^2$:
$$
\Vc_k=\{(a,b)\ |\ R_{a,b}^{\circ\, k-1}(\infty)=1\},\quad k=2,3,\dots
$$
These are complex one-dimensional slices of the
parameter space of quadratic rational functions.
These slices correspond to simple (periodic) types of behavior of one critical point
(note that $R_{a,b}(1)=\infty$ so that for all $(a,b)\in\Vc_k$, the critical point $\infty$ of
the function $R_{a,b}$ is periodic of period $k$).

We will identify pairs $(a,b)\in\Vc_k$ with the corresponding rational functions $R_{a,b}$.
For every $(a,b)\in\Vc_k$, let $\Omega_{a,b}$ denote the immediate basin of
the super-attracting fixed point $\infty$ of the rational function $R_{a,b}^{\circ k}$.
Define the set
$$
\Bc_k=\{(a,b)\in\Vc_k\ |\ 0\in R^{\circ m}_{a,b}(\Omega_{a,b}),\ m\ge 0\}.
$$
This set consists of all parameter values such that the critical point $0$
is in the immediate basin of the cycle of $\infty$.
Define the set
$$
\Lambda_k=\Vc_k-\overline{\Bc_k}.
$$
This is a one-dimensional complex manifold (for smoothness, see e.g. \cite{Stimson,Rees_b}).

Recall that a function $R_{a,b}\in\Lambda_k$ is hyperbolic of type C if
$R_{a,b}^{\circ m}(0)\in\Omega_{a,b}$ for some $m>0$.
The set of hyperbolic type C functions is open by Proposition \ref{dA}.

\subsection{Notation needed for the proof of Theorem \ref{typeC}}
\label{s_notation}
We will use the following notation throughout the proof of Theorem \ref{typeC}:
Let $H\subset\Lambda_k$ be a hyperbolic component of type C, and $f\in\d H$.
  Note that the boundary is taken in $\Lambda_k$, so that the boundaries of type B components
  are automatically excluded.
  Set $\Omega=\Omega_f$.
  Also, let $h$ be the center of the hyperbolic component $H$, i.e. the unique
  critically finite map in $H$.
There is a positive integer $k'$ such that for any parameter value $(a,b)\in H$,
  we have $R_{a,b}^{\circ k'}(0)\in\Omega_{a,b}$, and $k'$ is the minimal integer with this
  property.

Let $\Delta$ denote the unit disk $\{|z|<1\}$.
There is a holomorphic motion
$$
\mu:\Lambda_k\times\Delta\to\overline\C
$$
such that $\mu((a,b),z)$ is the point in $\Omega_{a,b}$, whose B\"ottcher coordinate is equal to $z$.
By Theorem \ref{mss_ext}, this holomorphic motion extends to a holomorphic motion
$$
\overline\mu:\Lambda_k\times\overline\Delta\to\overline\C
$$
By Proposition \ref{dA}, we have $f^{\circ k'}(0)\in\d\Omega$.
Therefore, $f(0)$ is on the boundary of some Fatou component $V$ containing a point $v$
such that $f^{\circ k'-1}(v)=\infty$.
Note that the particular choice of $V$ depends on $H$, not only on $f$
(there may be several possible choices if $f$ belongs to the boundaries of several type C components).

\subsection{Accessibility and non-recurrence}
The following Proposition shows that the Fatou components of $f$ do not
have topological pathologies (cf. \cite{A-Y}):

\begin{proposition}
\label{dlc}
  The boundary of $\Omega$ is locally connected.
  In particular, every boundary point is accessible from $\Omega$.
\end{proposition}

\begin{proof}
  Let $\lambda_0$ be the parameter value corresponding to $f$.
  Then the function $\overline\mu_f:z\mapsto\overline\mu(\lambda_0,z)$ is a quasi-symmetric
  homeomorphism between $\overline\Delta$ and $\overline\Omega$, by Theorem \ref{mss_ext}.
  It follows that the boundary of $\Omega$ is simply connected.
\end{proof}

Since $V$ is a pullback of $\Omega$, the boundary of $V$ is also simply connected.
In particular, $f(0)\in\d V$ is accessible from $V$.
The following proposition shows that the dynamical properties of $f$ are rather simple:

\begin{proposition}
\label{non-recurrent}
  The critical point $0$ of $f$ is non-recurrent.
\end{proposition}

\begin{proof}
  Note that all limit points of the orbit of $0$ belong to the forward orbit of $\d\Omega$ under $f$.
  Therefore, if $0$ is recurrent, then $0=f^{\circ l}(\d\Omega)$ for some $l=0,\dots,k-1$.
  In other terms, $0=f^{\circ l}(\overline\mu(\lambda_0,z))$, where $z$ is a point on the unit circle,
  and $\lambda_0$ is the parameter value corresponding to $f$.

  Consider the holomorphic function $\nu(\lambda)=R_\lambda^{\circ l}(\overline\mu(\lambda,z))$ over $\Lambda_k$.
  This function vanishes at the point $\lambda_0$.
  However, $\nu$ is not identically equal to zero, because e.g. for $h\in H$, the critical point is not in $\d\Omega_h$.
  Therefore, we can choose a small loop $\gamma:S^1\to\Lambda_k$ around $\lambda_0$ such that $\nu\circ\gamma$ loops around $0$.
  Take $\tilde z \in\Delta$ sufficiently close to $z$.
  Then the function $\tilde\nu(\lambda)=R_\lambda^{\circ l}(\overline\mu(\lambda,\tilde z))$ is
  uniformly close to $\nu$.
  In particular, $\tilde\nu\circ\gamma$ loops around $0$.
  Therefore, there exists a parameter value $\lambda_1$ inside $\gamma$ such that $\tilde\nu(\lambda_1)=0$.
  This means that $0$ lies in $R_{\lambda_1}^{\circ l}(\Omega_{\lambda_1})$, i.e. $\lambda_1\in\Bc_k$, a contradiction.
\end{proof}

\subsection{Restatement of Theorem \ref{typeC}}
\label{s_restate}
In this section, we will restate Theorem \ref{typeC}, and give
some details on the particular set of paths that was mentioned but not
defined in the statement of Theorem \ref{typeC}.
We also need to introduce some more notation.
Consider a simple path $\alpha_{-1}:[0,1]\to\overline V$ such that
$\alpha_{-1}(0)=f(0)$, $\alpha_{-1}(1)=v$, and $\alpha_{-1}(0,1]\subset V$.
The existence of such path follows from Proposition \ref{dlc}.
There is an $\alpha$-path $\alpha_0:S^1\to\overline\C$ such that
$f\circ\alpha_0(t_1,t_2)=\alpha_{-1}(|t_1|)$.
Let $\Ac$ be the set of all pullbacks of $\alpha_0$ under iterates of $f$, including $\alpha_0$.

\begin{proposition}
  The set of paths $\Ac$ is contracted.
\end{proposition}

\begin{proof}
  Since the critical point $0$ is non-recurrent, there exists a neighborhood $U$ of $\alpha_0(S^1)$
  that does not intersect the post-critical set.
  Choose a smaller neighborhood $U'$ of $\alpha_0(S^1)$ that is compactly contained in $U$.
  By the Koebe distortion theorem, the set of pullbacks of $U'$ is contracted.
  It follows that $\Ac$ is contracted.
\end{proof}

We can now conclude that there exists a regluing $\Phi$ of $\Ac$ into some
contracted set of disjoint $\beta$-paths.
Moreover, the function $g=\Phi\circ f\circ\Phi^{-1}$ is well-defined as a topological model.
We can now give a precise statement, from which Theorem \ref{typeC} follows:

\begin{theorem}
\label{reg_f}
  The map $g$ is equivalent as a topological model to the
  critically finite hyperbolic rational function $h\in H$.
\end{theorem}

Theorem \ref{reg_f} implies Theorem \ref{typeC}.
Indeed, if $g=\Phi\circ f\circ\Phi^{-1}$ is conjugate to $h$, then $h=\Psi\circ f\circ\Psi^{-1}$
for some topological regluing $\Psi$.
It follows that $f=\Psi^{-1}\circ h\circ\Psi$.
Note that $\Psi^{-1}$ is also a regluing.

Thus it remains to prove Theorem \ref{reg_f}.
We first prove that $g$ is Thurston equivalent to $h$.
By a theorem of Mary Rees \cite{Rees_description}, Thurston equivalence to a hyperbolic
rational function implies semi-conjugacy.
We will recall the proof of this theorem.
What remains is to prove that all fibers are trivial.
Several ideas for this part were taken from \cite{Rees_description}.
Overall, the argument is rather simple, but we need to know from the very beginning that $g$
is well-defined as a topological ramified covering.
Here we are using results of Section \ref{s_moore} on the existence of topological regluing.

\section{The proof of Theorem \ref{typeC}}
\label{s_proof}

{\footnotesize
In this section, we prove Theorem \ref{typeC}.
In several important places, the argument is inspired by description \cite{Rees_description} of topological models for
hyperbolic quadratic rational functions.
We use notation introduced in Sections \ref{s_restate} and \ref{s_notation}.
}

\subsection{Backward stability}
Recall the following theorem of Ma\~ne \cite{Mane,TanShi}:

\begin{theorem}[Backward stability]
 \label{mane}
 Let $F$ be a rational function with the Julia set $J$.
 Take any $\eps>0$.
 Then for every $x\in J$ that is not a parabolic periodic point, and that is not in the $\omega$-limit
 set of a recurrent critical point, there exists a neighborhood $U$ of $x$ such that
 \begin{enumerate}
   \item for all $n$, the diameter of any component of $F^{-n}(U)$ does not
   exceed $\eps$ in the spherical metric,
   \item for every $\eps'>0$, there exists a positive integer $n_0$ such that
   for $n>n_0$, every component of $F^{-n}(U)$ has diameter $\le\eps'$.
 \end{enumerate}
\end{theorem}

If $F=f$, then the assumptions of this theorem are satisfied for all points in the Julia set.
From the backward stability of $f$ on $J_f$, we can deduce the backward stability of $g$ on $J_g$:

\begin{proposition}[Backward stability of $g$ on $J_g$]
\label{bs}
 For every point $x\in J_g$ and every $\eps>0$, there exists a neighborhood $U$ of $x$ such that
 \begin{enumerate}
   \item for all $n$, the diameter of any component of $g^{-n}(U)$ is less than
   $\eps$ in the spherical metric,
   \item for every $\eps'>0$, there exists a positive integer $n_0$ such that
   for $n>n_0$, every component of $g^{-n}(U)$ has diameter $<\eps'$.
 \end{enumerate}
\end{proposition}

\begin{proof}
 Consider a point $z\in J_f$.
 Define a {\em distinguished neighborhood} of $z$ as a Jordan domain $W_z$ containing $z$
 and satisfying the following properties:
 \begin{itemize}
   \item if $z\not\in\Im\Ac$, then the boundary of $W_z$ is disjoint from $\Im\Ac$.
   \item if $z=\alpha(0,1)$ for some $\alpha\in\Ac$, then the intersection of $W_z$ with
   $\alpha(S^1)$ is $\{\alpha(t_1,t_2)\ |\ |t_1|<\tau\}$ for some $\tau>0$.
 \end{itemize}
 For any set $A$, we define $\hat\Phi(A)$ as the set of all points $\Phi(a)$, where
 $a\in A-\Im\Ac$, together with all points $\beta(t_1,t_2)$, where $\beta\in\Bc$ is the $\beta$-path
 corresponding to some $\alpha$-path $\alpha\in\Ac$, and $\alpha(t_1,t_2)\in A$.
 For any distinguished neighborhood $W_z$, the set $\hat\Phi(W_z)$ is either a
 Jordan neighborhood of $\Phi(z)$, if $z\not\in\Im\Ac$, or a pair of disjoint Jordan neighborhoods around
 $\beta(0,\pm 1)$ if $z=\alpha(0,\pm 1)$ for $\alpha\in\Ac$.

 We know (see Section \ref{s_moore}) that every point $z\in J_f$ has a basis of distinguished neighborhoods.
 Let $\Wc_f$ denote the set of all distinguished neighborhoods $W_z$ around all points $z\in J_f$.
 In particular, the set $\{W\cap J_f\ |\ W\in\Wc_f\}$ is a basis of the topology in $J_f$.
 Now define $\Wc_g$ as the set of connected components of $\hat\Phi(W)$, where $W$ runs though all elements of $\Wc_f$.
 Then the set $\{E\cap J_g\ |\ E\in\Wc_g\}$ is a basis of the topology in $J_g$.

 Take $\eps,\eps'>0$ and $x\in J_g$.
 Define $\Ec_g$ to be the set of all $E\in\Wc_g$ such that the diameter of $E$ is less than $\eps$.
 Also, consider $\Ec'_g$, the set of all $E\in\Wc_g$ such that the diameter of $E$ is less than $\eps'$.
 Then $\Ec_g$ and $\Ec'_g$ are open coverings of $J_g$.
 Let $\Ec_f$ be the set of all $W\in\Wc_f$ such that all components of $\hat\Phi(W)$ belong to $\Ec_g$
 (there are at most two components).
 It is not hard to see that $\Ec_f$ is an open covering of $J_f$.
 The open covering $\Ec'_f$ is defined similarly, with $\Ec'_g$ replacing $\Ec_g$.
 Consider any preimage $z$ of $x$ under $\hat\Phi$ (there can be at most two preimages).
 By the backward stability of $f$, there is a neighborhood $U_f$ of $z$
 depending on $\eps$ but not on $\eps'$, with the following properties:
 \begin{enumerate}
   \item for any $n>0$, every connected component of $f^{-n}(U_f)$ lies entirely in some element of $\Ec_f$;
   \item there exists $n_0$ (depending on both $\eps$ and $\eps'$)
   such that for all $n>n_0$, every connected component of $f^{-n}(U_f)$
   lies entirely in some element of $\Ec'_f$.
 \end{enumerate}
 Moreover, we can assume that $U_f\in\Wc_f$.
 Let $U_g$ be the connected component of $\hat\Phi(U_f)$ containing $x$.
 Then:
 \begin{enumerate}
   \item For any $n>0$, every connected component of $g^{-n}(U_g)$ lies entirely in some element of $\Ec_g$.
   Indeed, every connected component of $g^{-n}(U_g)$ is a connected component of $\hat\Phi(W)$,
   where $W$ is a connected component of $f^{-n}(U_f)$.
   Since $W$ lies in a single element $E$ of $\Ec_f$, and all (one or two) components of $\hat\Phi(E)$
   lie in $\Ec_g$, all components of $\hat\Phi(W)$ lie in $\Ec_g$.
   In particular, every component of $g^{-n}(U_g)$ lies in $\Ec_g$.
   \item There exists $n_0$ (depending on both $\eps$ and $\eps'$) such that for all $n>n_0$,
   every connected component of $g^{-n}(U_g)$ lies in some element of $\Ec_g$.
   The proof is similar.
 \end{enumerate}
 This concludes the proof of the theorem.
\end{proof}

\subsection{Thurston equivalence}
\label{s_Th_eq}
The map $g$ is critically finite.
Indeed, the critical value of $g$ is $\Phi(v)$, and its forward orbit under $g$ coincides with
the $\Phi$-image of the forward orbit of $v$ under $f$.
In this section, we will prove that the map $g$ is Thurston equivalent to $h$.

We first recall the notion of Thurston equivalence.
A ramified self-covering of the sphere with finite post-critical set is sometimes called a {\em Thurston map}
(recall that the post-critical set is the union of forward orbits of all critical values, including the critical
values but, in general, not the critical points).
Two Thurston maps $F$ and $G$ with post-critical sets $P_F$ and $P_G$, respectively, are
called {\em Thurston equivalent} if there exist homeomorphisms $\phi,\psi:S^2\to S^2$ that
make the following diagram commutative
$$
\begin{CD}
  (S^2,P_F) @>F>> (S^2,P_F)\\
  @V\psi VV @VV\phi V\\
  (S^2,P_G) @>G>> (S^2,P_G)
\end{CD}
$$
and such that $\phi$ and $\psi$ are homotopic relative to $P_F$ through ramified
self-coverings, in particular, $\phi|_{P_F}=\psi|_{P_F}$.
The following are well-known useful criteria of Thurston equivalence:

\begin{theorem}
\label{def-teq}
  Suppose that $F_t$, $t\in [0,1]$, is a continuous family of Thurston maps of degree 2
  such that the size of $P_{F_t}$ does not change with $t$.
  Then $F_0$ is Thurston equivalent to $F_1$.
\end{theorem}

\begin{theorem}
\label{cut-teq}
  Let $Z$ be a compact connected locally connected non-separating subset of $S^2$.
  Suppose that quadratic Thurston maps $F$ and $G$ are such that $F=G$ on $Z$ and $P_F=P_G\subset Z$.
  Assume additionally that $F$ and $G$ coincide on some neighborhoods of critical points of $F$
  (in particular, the critical points of $F$ and $G$ are the same).
  Then $F$ and $G$ are Thurston equivalent.
\end{theorem}

For completeness, we sketch the proofs of these theorems.

\begin{proof}[Proof of Theorem \ref{def-teq}]
  It suffices to prove that $F_t$ is Thurston equivalent to $F_{t'}$ provided that $t$ is close to $t'$.
  The ramified coverings $F_t$ and $F_{t'}$ are then uniformly close, therefore, their post-critical sets are also close
  (this follows from the fact that the (weighted) number of critical points in a disk can be computed as a winding number).
  In particular, these maps are conjugate on their post-critical sets.
  Moreover, we can choose a homeomorphism $\phi:S^2\to S^2$ that is close to the identity, maps $P_{F_t}$ to $P_{F_{t'}}$
  and conjugates the dynamics of $F_t$ on $P_{F_t}$ with the dynamics of $F_{t'}$ on $P_{F_{t'}}$.
  The multivalued function $F_{t'}^{-1}\circ\phi\circ F_t$ has two single valued branches.
  Indeed, every critical value of $\phi\circ F_t$ is a critical value of $F_{t'}$, and the only $F_{t'}$-preimage
  of this critical value is the corresponding critical point, since the map is quadratic.
  Thus, whenever $\phi(F_t(z))$ is a ramification point of $F_{t'}^{-1}$, the function $F_{t'}^{-1}\circ\phi\circ F_t$
  can be written as $\sqrt{u^2}$ in some local coordinates.
  It follows that this function has no ramification points, hence it splits into two single valued branches.
  One of the branches is close to the identity --- denote this branch by $\psi$.
\end{proof}

\begin{proof}[Proof of Theorem \ref{cut-teq}]
  By the same argument as in the proof of Theorem \ref{def-teq}, the multivalued function $G^{-1}\circ F$
  splits into two single valued branches.
  Near the critical points, one of the branches of $G^{-1}\circ F$ is the identity.
  It follows that there is a branch $\phi$ of $G^{-1}\circ F$ that restricts to the identity on $Z$ ---
  the branches cannot switch outside of the critical points, and $z\in G^{-1}(F(z))$ for all $z\in Z$.
  There exists a continuous surjective map $\gamma:\overline\Delta\to S^2$ that establishes a holomorphic homeomorphism
  between $\Delta$ and $S^2-Z$ (the set $Z$ is infinite because any quadratic map has exactly two critical values).
  It is not hard to see (with the help of Carath\'eodory's theory) that $\gamma^{-1}\circ\phi\circ\gamma$ extends continuously to a
  self-homeomorphism of $\overline\Delta$ that is the identity on the unit circle.
  Such homeomorphism is isotopic to the identity transformation of $\overline\Delta$ through
  self-homeomorphisms of $\overline\Delta$ restricting to the identity on the unit circle.
  It follows that $\phi$ is isotopic to the identity through homeomorphisms of $S^2$, whose
  restrictions to $Z$, in particular, to the post-critical set $P_F=P_G$, are the identity.
  The theorem follows.
\end{proof}

Consider a simple curve $t\mapsto (a_t,b_t)$, $t\in [0,1]$, in the parameter space
$\Lambda_k$ that connects $f$ to $h$
and lies entirely in $H$ except for the starting point: $R_{a_0,b_0}=f$ and $R_{a_1,b_1}=h$.
Set $f_t=R_{a_t,b_t}$.
Define a continuous function $u_t$ by the following properties:
$$
u_1=0,\quad f_t^{\circ k'}(\sqrt{u_t})=\infty
$$
(observe that the left-hand side of the second equality does not depend on the choice of the square root, and
is a polynomial of $u_t$).
Then we have $u_0=\alpha_0(1,0)^2$.
For all $t\in [0,1]$, consider the multivalued function
$$
j_t=\sqrt{z^2-u_t}
$$

Let us also consider a continuous family of $\alpha$-paths $\alpha_t$ such that, for $t=0$,
we have the same $\alpha_0$ as above, and
$$
\alpha_t(t_1,t_2)=-\alpha_t(-t_1,t_2),\quad \alpha_t(\pm 1,0)^2=u_t,\quad\alpha_1\equiv 0
$$
We can also assume that $f_t\circ\alpha_t(S^1)$ lies entirely in a Fatou component of $f_t$,
with the only exception that the center $\alpha_0(0,\pm 1)$ of the very first path lies in the Julia set.
Set $v_t=f_t(\sqrt{u_t})$ (this point is independent of the choice of the square root, and is a fractional
linear function of $u_t$).
Note that $v_0=v$.
Note also that for all $t\in [0,1]$, we have $v_t\not\in\alpha_t(S^1)$.
Indeed, $\pm\sqrt{u_t}$ are the endpoints of $\alpha_t$.
They belong to strictly pre-periodic Fatou components of $f_t$ (actually, to the same Fatou component
unless $t=0$), therefore, their images under $f_t$ cannot lie in $\alpha_t(S^1)$.
It follows that we can choose a unique branch of $j_t$ such that $j_t(v_t)$ depends continuously on $t$ and $j_1=id$.
From now on, we write $j_t$ to mean this particular branch only.

Note that the function $f_t\circ j_t^{-1}$ extends to the Riemann sphere as a quadratic
rational function $Q_t$.
Set $q_t=j_t\circ Q_t$.
This function is defined on the complement to $Q_t^{-1}(\alpha_t(S^1))$, i.e.
on the complement to a pair of disjoint simple curves.
Intuitively, the function $q_t$ is obtained from $f_t$ by regluing the path $\alpha_t$, or,
to be more precise, we have $q_t=j_t\circ f_t\circ j_t^{-1}$ wherever the right-hand side is
defined (however, $q_t$ is defined on a bigger set).
Note that $q_t$ is always defined and analytic in a neighborhood of $0$, because
$Q_t(0)=v_t$ avoids the image of $\alpha_t$.
In particular, $0$ is a critical point of $q_t$.
Another critical point is $j_t(\infty)=\infty$.
Moreover, the critical orbits of $q_t$ are finite and of constant size.
More precisely, we have
$$
q_t^{\circ\,k'}(0)=j_t\circ f_t^{\circ\, k'-1}(v_t)=\infty,\quad q_t^{\circ\, k}(\infty)=j_t\circ f_t^{\circ\, k}(\infty)=\infty.
$$
In other words, the orbits of $0$ and $\infty$ under $q_t$ have the same dynamics as the
orbits of $0$ and $\infty$, respectively, under $h$.
The relations given above can be easily proved using that $q_t=j_t\circ Q_t$ and that
$Q_t\circ j_t=f_t$ wherever the left-hand side is defined.
The construction of the map $q_t$ will reappear in Section \ref{s_hr}, where more details can be found.

The maps $q_t$ are critically finite but they are not Thurston maps because of discontinuities.
However, one can approximate $q_t$ by ramified coverings $\hat q_t$ that differ from $q_t$
only in a small neighborhood of $Q_t^{-1}(\alpha_t(S^1))$.
Moreover, we can choose these approximations to vary continuously with $t$.
Additionally, we can arrange that $\hat q_1=h$; in any case, $\hat q_1$ is Thurston equivalent to $h$.
Thus $\hat q_t$ is a continuous family of Thurston maps, whose post-critical sets are of constant size.
By Theorem \ref{def-teq}, we obtain

\begin{lemma}
  \label{hatq_h}
  The map $\hat q_0$ is Thurston equivalent to $\hat q_1=h$.
\end{lemma}

What remains to prove is the following

\begin{lemma}
  The maps $\hat q_0$ and $g$ are Thurston equivalent.
\end{lemma}

\begin{proof}
  We have
  $$
  g(x)=\Psi\circ q_0\circ\Psi^{-1}(x)
  $$
  for all $x\not\in\Im\Bc$, where $\Psi=\Phi\circ j_0^{-1}$.
  Note that the post-critical set of $g$ is disjoint from $\Im\Bc$.
  Recall that $\Phi$ reglues the set of $\alpha$-paths $\Ac$ in the dynamical sphere of $f$
  into a set of $\beta$-paths $\Bc$ in the dynamical sphere of $g$.
  Let $\beta_0\in\Bc$ be the $\beta$-path corresponding to the $\alpha$-path $\alpha_0$.

  Consider a continuous map $\gamma:\overline\Delta\to S^2$ with the following properties:
  \begin{itemize}
    \item The image $\gamma(S^1)$ is a simple curve, does not separate the sphere,
    contains the post-critical set of $g$, and does not intersect $\Im(\Bc)$.
    \item The restriction of $\gamma$ to $\Delta$ is a holomorphic homeomorphism between $\Delta$ and $S^2-\gamma(S^1)$.
  \end{itemize}
  The existence of a simple curve containing the post-critical set of $g$ and disjoint from $\Im\Bc$
  follows from a simple Baire category argument, see Section \ref{s_regs} for more detail.
  Then we can define a homeomorphism $\hat\Psi:S^2\to S^2$ such that $\hat\Psi^{-1}=\Psi^{-1}$ on $\gamma(S^1)$
  and on $g\circ\gamma(S^1)$.
  The restriction of $g\circ\gamma$ to the unit circle will be the same as the
  restriction of $\hat\Psi\circ q_0\circ\hat\Psi^{-1}\circ\gamma$ to the unit circle.
  Indeed, for every $z\in S^1$, we have
  $$
  \hat\Psi\circ q_0\circ\hat\Psi^{-1}\circ\gamma(z)=\hat\Psi\circ\Psi^{-1}\circ\Psi\circ q_0\circ\Psi^{-1}\circ\gamma(z)=
  $$
  $$
  =\hat\Psi\circ\Psi^{-1}\circ g\circ\gamma(z)=g\circ\gamma(z).
  $$
  We can even find a continuous approximation $\hat q_0$ of $q_0$ such that the restriction of
  $\hat\Psi\circ\hat q_0\circ\hat\Psi^{-1}\circ\gamma$ to the unit circle is still the same.
  Indeed, we must have $q_0=\hat q_0$ on the set $\Psi^{-1}\circ\gamma(S^1)$, which disjoint
  from the discontinuity locus of $q_0$.
  By Theorem \ref{cut-teq}, $g$ is Thurston equivalent to $\hat\Psi\circ\hat q_0\circ\hat\Psi^{-1}$, hence to $\hat q_0$.
\end{proof}

\subsection{Semi-conjugacy}
Recall the following theorem of Mary Rees \cite{Rees_description}:

\begin{theorem}
\label{semi-conj}
  Suppose that a Thurston map $F$ of degree 2 is Thurston equivalent to a hyperbolic rational function $G$.
  Moreover, suppose that there is an $F$-invariant complex structure near the critical orbits of $F$.
  Then $F$ is semi-conjugate to $G$, i.e. there is a continuous map $\phi$ from
  the dynamical sphere of $F$ to the dynamical sphere of $G$ such that $\phi\circ F=G\circ\phi$.
\end{theorem}

\begin{proof}
We can assume $F$ to be defined on $\overline\C$
(i.e. on a sphere with a global complex structure) and holomorphic on some open set $U$ containing
the post-critical set and satisfying $F(U)\Subset U$.
We have the diagram
$$
\begin{CD}
\overline\C @>F>> \overline\C\\
@VV\phi_1V @VV\phi_0V\\
\overline{\C} @>G>> \overline{\C}
\end{CD}
$$
where $\phi_0$ and $\phi_1$ are homeomorphisms holomorphic on $U$; moreover,
$\phi_1=\phi_0$ on $U$, and $\phi_0$ is isotopic to $\phi_1$ relative to $U$.
Consider the multivalued function $G^{-1}\circ \phi_1\circ F$.
Since the critical values of $\phi_1\circ F$ coincide with ramification points of $G^{-1}$,
this multivalued function has a single valued branch $\phi_2$ such that $\phi_2=\phi_1$
on $F^{-1}(U)$.
We used that $F$ and $G$ have degree 2, because we need that only critical points can
map to critical values.
We have the following diagram:
$$
\begin{CD}
\overline{\C} @>F>> \overline{\C} @>F>> \overline{\C}\\
@VV\phi_2V @VV\phi_1V @VV\phi_0V\\
\overline{\C} @>G>> \overline{\C} @>G>> \overline{\C}
\end{CD}
$$
Similarly, we can define a sequence of homeomorphisms $\phi_n$ with the following properties:
$G\circ \phi_n=\phi_{n-1}\circ F$ and $\phi_n=\phi_{n-1}$ on $F^{-(n-1)}(U)$.
Then $\phi_n$ is a single valued branch of $G^{-n}\circ\phi_0\circ F^{\circ n}$.

We would like to prove that the sequence of maps $\phi_n$ converges uniformly.
This would follow from the estimate
$$
d(\phi_{n+1}(x),\phi_n(x))\le Cq^{-n},
$$
where $0<q<1$ is a number independent of $x$.
We can assume that $x\not\in F^{-n}(U)$, otherwise the left-hand side is zero.
Consider a curve $\gamma$ connecting $\phi_0\circ F^{\circ n}(x)$ with $\phi_1\circ F^{\circ n}(x)$ in $\overline\C-\overline U$
(we can arrange that this open set be connected by choosing a smaller $U$ if necessary).
Since $\overline\C-\overline U$ is compactly contained in $\overline\C-\overline{F(U)}$,
the hyperbolic length of $\gamma$ in $\overline\C-\overline{F(U)}$ can be made bounded by
some constant independent of $x$ and $n$.
The length of the pull-back of $\gamma$ under $G^{\circ n}$ is bounded by $Cq^{-n}$ with $0<q<1$
by the Poincar\'e distance argument.
The desired estimate now follows.

Let $\phi$ denote the limit of $\phi_n$.
Passing to the limit in both sides of the equality $G\circ\phi_n=\phi_{n-1}\circ F$,
we obtain that $G\circ \phi=\phi\circ F$.
\end{proof}

We also need the following general fact:

\begin{proposition}
\label{connected}
Suppose that a continuous map $\phi:S^2\to S^2$ is the
limit of a uniformly convergent sequence of homeomorphisms $\phi_n$.
Then, for any point $v\in S^2$, the fiber $\phi^{-1}(v)$ is connected.
\end{proposition}

\begin{proof}
Indeed, consider two points $z$ and $w$ in the fiber $\phi^{-1}(v)$.
For large $n$, the points $\phi_n(z)$ and $\phi_n(w)$ are very close to each other.
Let $D_n$ be a small closed disk containing both of them, and set $A_n=\phi_n^{-1}(D_n)$.
We can assume that the diameter of $D_n$ tends to 0 as $n\to\infty$.
The sequence of compact sets $A_n$ has a subsequence that converges in the Hausdorff metric.
Denote the limit by $A$.
As a Hausdorff limit of compact connected sets, the set $A$ is connected.
Moreover, it contains both points $z$ and $w$.
We claim that $\phi(A)=v$.
Indeed, for any point $a\in A$, there is a sequence $a_n\in A_n$ such that $a_n\to a$.
The distance between $\phi_n(a_n)$ and $\phi(a)$ tends to zero, and $\phi_n(a_n)\to v$, hence $\phi(a)=v$.
Thus any pair of points in $\phi^{-1}(v)$ belongs to a common connected subset of $\phi^{-1}(v)$.
This means that $\phi^{-1}(v)$ is connected.
\end{proof}

Thus, in our setting, we obtain the following

\begin{theorem}
\label{h-semiconj-g}
  The map $h$ is semi-conjugate to $g$, i.e. there is a continuous surjective map $\phi:S^2\to\overline\C$
  such that $h\circ\phi=\phi\circ g$.
  Moreover, the fibers of $\phi$ are connected.
\end{theorem}

\subsection{Backwards stability II}
In this section, we will discuss certain consequences of Proposition \ref{bs}.
Let $J_g$ be the Julia set of $g$.
As before, we use the symbol $d$ to denote the spherical metric.

\begin{proposition}
\label{delta0}
  There is a positive real number $\delta_0$ such that for every pair $x,y\in J_g$
  with $d(x,y)<\delta_0$, we have $g(x)\ne g(y)$.
\end{proposition}

\begin{proof}
  Since there are no critical points in $J_g$, every point $x\in J_g$ has a
  neighborhood $U_x$ such that $g|_{U_x}$ is injective.
  Let $\delta_0$ be the Lebesgue number of the covering $\Uc=\{U_x\}_{x\in J_g}$.
  Now assume that $d(x,y)<\delta_0$.
  Then $x,y\in U$, where $U\in\Uc$.
  Since $g$ is injective on $U$, the result follows.
\end{proof}

\begin{proposition}
\label{n0}
  For every $\eps>0$ and $\eps'>0$, there exist a positive real number $\delta(\eps)$ (depending only on $\eps$),
  and a positive integer $n_0(\eps,\eps')$ (depending on $\eps$ and $\eps'$) such that for $x,y\in J_g$
  \begin{enumerate}
    \item if $d(x,y)<\delta(\eps)$, then for every $n\ge 0$
    and for every $x'$ such that $g^{\circ n}(x')=x$, there is a point $y'$ such that
    $g^{\circ n}(y')=y$, and $d(x',y')<\eps$,
    \item if $d(x,y)<\delta(\eps)$ and $n\ge n_0(\eps,\eps')$, then for every $x'$
    such that $g^{\circ n}(x')=x$, there is a point $y'$ such that $g^{\circ n}(y')=y$, and $d(x',y')<\eps'$.
  \end{enumerate}
\end{proposition}

\begin{proof}
  1. The proof is based on the backward stability of $g$ on $J_g$ (Proposition \ref{bs}).
  Let $U_{x,\eps}$ be the neighborhood $U$ from Proposition \ref{bs}: it depends on $x\in J_g$ and $\eps>0$.
  The covering $\{U_{x,\,\eps}\}_{x\in J_g}$ has a finite subcovering $\Uc$.
  Let $\delta(\eps)$ be a Lebesgue number of $\Uc$.
  Assume that $d(x,y)<\delta(\eps)$.
  Then there is a neighborhood $U_{z,\,\eps}\in\Uc$ containing both $x$ and $y$.
  Take $x'$ such that $g^{\circ n}(x')=x$.
  Let $y'$ be the point in the connected component of $g^{-n}(U_{z,\,\eps})$ containing $x'$
  such that $g^{\circ n}(y')=y$.
  Then we have $d(x',y')<\eps$.

  2. For a point $x\in J_g$, let $\nu(x)$ be the positive integer such that for all
  $n\ge\nu(x)$, every component of $g^{-n}(U_{x,\,\eps})$ has diameter $<\eps'$.
  The existence of $\nu(x)$ follows from Proposition \ref{bs}.
  Define $n_0(\eps,\eps')$ as the maximum of $\nu(x)$ over all neighborhoods $U_{x,\,\eps}\in\Uc$.
  Assume that $d(x,y)<\delta(\eps)$ and $n\ge n_0(\eps,\eps')$.
  Then we have $d(x',y')<\eps'$ for $x'$ and $y'$ chosen as in part 1.
\end{proof}

\subsection{Triviality of fibers}
In this section, we prove that the continuous map $\phi$ from Theorem \ref{h-semiconj-g} is actually a homeomorphism.
Note first that the restriction of $\phi$ to every Fatou component of $g$ is injective.
Thus the fibers of $\phi$ lie entirely in $J_g$.
Moreover, we know that fibers are connected.
Assume that there exists at least one non-trivial fiber $Z$.
As a connected set containing more than one point, the set $Z$ is infinite.

\begin{lemma}
\label{invol}
  There is a self-homeomorphism $M\ne id$ of the dynamical sphere of $g$ such that
  $M^{\circ 2}=id$, and $g\circ M=g$.
  Moreover, we have $\phi\circ M=-\phi$.
\end{lemma}

\begin{proof}
  Indeed, the multivalued map $g^{-1}\circ g$ splits into two single valued branches.
  One of these branches is the identity transformation.
  Let $M$ be the other branch.
  We have $g\circ M=g$ by definition.
  It follows that $g\circ M^{\circ 2}=g$, thus $M^{\circ 2}$ is either $M$ or the identity.
  Since $M\ne id$, we have $M^{\circ 2}\ne M$.
  Therefore, $M^{\circ 2}=id$.

  We have
  $$
  h\circ \phi\circ M=\phi\circ g\circ M=\phi\circ g=h\circ\phi,
  $$
  therefore $\phi(M(x))=\pm\phi(x)$ for all $x$ in the dynamical sphere of $g$.
  Near the critical points of $g$ (which are the $\phi$-preimages of $0$ and $\infty$),
  the map $\phi$ is a homeomorphism, and we have the minus sign.
  By continuity, the sign is the same for all points.
\end{proof}

\begin{proposition}
\label{fibinj}
  The restriction of $g$ to $Z$ is injective.
\end{proposition}

\begin{proof}
  Assume the contrary: $g(x)=g(y)$ for a pair of different points $x,y\in Z$.
  Let $M$ be the involution introduced in Lemma \ref{invol}.
  Then $y=M(x)$.
  For any point $z\in Z$, we have $\phi(x)=\phi(z)$, therefore, $\phi(y)=\phi(M(z))$, and $M(z)\in Z$.
  Thus $M(Z)=Z$.
  Applying $\phi$ to both sides of this equality, we obtain $\phi(Z)=\phi(M(Z))=-\phi(Z)$.
  Recall that $\phi(Z)$ is just a point.
  It follows that $\phi(Z)=\{0\}$ or $\{\infty\}$, a contradiction.
\end{proof}

The following argument is a modification/generalization of one from \cite{Rees_description}.

\begin{proposition}
\label{closeim}
 Let $\delta_0$ be the number introduced in Proposition \ref{delta0}.
 For every $\eps>0$, there exist $x,y\in J_g$ such that
 $$
  d(x,y)\ge\delta_0/2,\quad d(g(x),g(y))<\eps,\quad \phi(x)=\phi(y).
 $$
\end{proposition}

Recall our assumption that $\phi$ has a nontrivial fiber $Z$.
The statement of the proposition depends on this assumption, although $Z$ was not mentioned at all.

\begin{proof}
  Take $\eps=\delta_0/2$.
  Let $N$ be a positive integer such that every subset of $J_g$ with at least $N$ points
  has a pair of distinct points on distance $<\delta(\eps)$ (where $\delta(\eps)$ is
  as in Proposition \ref{n0}).
  Consider a finite subset $Z_0$ of $Z$ containing $N$ points.
  Define $\eps'$ to be the minimal distance between different points in $Z_0$.
  Set $n_0=n_0(\eps,\eps')$ (see Proposition \ref{n0}).
  The set $f^{\circ n_0}(Z_0)$ has cardinality $N$ (by Proposition \ref{fibinj}), therefore, there is a pair of points
  $x_{n_0}$, $y_{n_0}$ in this set such that $d(x_{n_0},y_{n_0})<\delta(\eps)$.
  Note that $\delta(\eps)\ge\eps$, in particular, $d(x_{n_0},y_{n_0})<\eps$.

  For every $k=0,\dots,n_0$, we can define $x_k$ and $y_k$ inductively by the following relations:
  $$
   x_k,y_k\in f^{\circ k}(Z_0),\quad f(x_k)=x_{k+1},\quad f(y_k)=y_{k+1}.
  $$
  Then either $y_k$ is the closest to $x_k$ preimage of $y_{k+1}$ and,
  in particular, $d(x_k,y_k)<\eps$, or $d(x_k,y_k)\ge\delta_0-\eps=\delta_0/2$.
  Indeed, let $y'_k$ be the closest to $x_k$ preimage of $y_{k+1}$.
  If $y'_k\ne y_k$, then
  $$
   d(x_k,y_k)\ge d(y_k,y'_k)-d(y'_k,x_k)\ge\delta_0-\eps.
  $$
  Suppose first that $d(x_k,y_k)\ge\delta_0/2$ for some $k$.
  Then $x=x_k$ and $y=y_k$ have the required properties.
  Suppose now that $y_k$ is the closest to $x_k$ preimage of $y_{k+1}$, for all $k$.
  Then, in particular, we have $d(x_0,y_0)<\eps'$.
  But this contradicts our choice of $\eps'$, since $x_0,y_0\in Z_0$.
\end{proof}

We will now get a contradiction with our assumption that $Z$ is non-trivial.
Indeed, let $x_n$ and $y_n$ be points in $J_g$ such that
$$
d(x_n,y_n)\ge\delta_0/2,\quad d(g(x_n),g(y_n))<1/n,\quad \phi(x_n)=\phi(y_n).
$$
The existence of such points follows from Proposition \ref{closeim}.
Choosing suitable subsequences if necessary, we can arrange that both $x_n$ and $y_n$ converge
to $x$ and $y$, respectively.
For these limit points, we must have
$$
d(x,y)\ge\delta_0/2,\quad g(x)=g(y),\quad \phi(x)=\phi(y).
$$
In other terms, $x$ and $y$ belong to the same fiber of $\phi$, and
the restriction of $g$ to this fiber is not injective.
This contradicts Proposition \ref{fibinj}.
The contradiction shows that all fibers of $\phi$ are trivial, thus $\phi$ is a homeomorphism.

This finishes the proof of Theorem \ref{reg_f} and Theorem \ref{typeC}.

\section{Holomorphic regluing}
\label{s_hr}

{\footnotesize In this section, we just sketch main characters of
holomorphic theory of regluing, and give the most basic constructions.}

\subsection{An example}
\label{s_ex}
We consider first a simple example, where we define an explicit sequence of approximations to a regluing.
This sequence will consist of partially defined but holomorphic functions.

Let $f$ be the quadratic polynomial $z\mapsto z^2-6$.
The Julia set $J_f$ of $f$ is a Cantor set that lies on the real line.
Recall that the biggest component of the complement to $J_f$ in $[-3,3]$ is $(-\sqrt{3},\sqrt{3})$.
Suppose we want to reglue the segment $[-\sqrt{3},\sqrt{3}]$, thus connecting two parts of $J_f$.
This is done by the following map:
$$
j(z)=\sqrt{z^2-3}
$$
(which is understood as a branch over the complement to $[-\sqrt 3,\sqrt 3]$ that is tangent to
the identity at infinity).
The inverse map is given by the formula $j^{-1}(z)=\sqrt{z^2+3}$, and is defined on the complement to $[-i\sqrt{3},i\sqrt{3}]$.

Consider the composition $f\circ j^{-1}$.
It turns out that this function extends to a quadratic polynomial!
Indeed, we have
$$
f(\sqrt{z^2+3})=(z^2+3)-6=z^2-3.
$$
We denote the polynomial in the right-hand side by $p_{-3}$; in general, $p_c$
stands for the quadratic polynomial $z\mapsto z^2+c$.
From a more conceptual viewpoint, it suffices to see that $f\circ j^{-1}$ extends
continuously to the segment $[-i\sqrt{3},i\sqrt{3}]$.
This follows from the fact that $f$ folds the segment $[-\sqrt{3},\sqrt 3]$ at $0$.
It is a good news that $f\circ j^{-1}$ is a quadratic polynomial.
However, it is very difficult to see the dynamics of a composition, even if the dynamics of the
both factors is well-understood.

In fact, what we really want to consider is not the composition but the ``conjugation'' $j\circ f\circ j^{-1}$.
Since $f\circ j^{-1}$ extends to a polynomial, we define the new function $f_1$ as $j\circ p_{-3}$,
which is defined on a larger set than $j\circ f\circ j^{-1}$.
A bad news is that the function $f_1$ is not continuous.
The discontinuity of this function is due to the discontinuity of $j$.
Actually, the function $f_1$ is defined and holomorphic on the complement to two simple curves ---
the pullbacks of $[-\sqrt 3,\sqrt 3]$ under $p_{-3}$, and it ``reglues'' these curves in a sense.
We would like to get rid of this discontinuity by ``conjugating'' $f_1$ with yet another regluing map.
To this end, we need an injective holomorphic function $j_1$ defined on the domain of $f_1$
and having the same type of discontinuity at the two special curves, where $f_1$ is undefined.
We cannot take $j_1=f_1$ because $f_1$ is, in general, two-to-one
(it is the square root of a degree four polynomial).
However, we can take
$$
j_1=\sqrt{f_1-f_1(0)}.
$$
The square root may look disturbing but it does not actually create any ramification,
so that the function $j_1$ is not a multivalued function, it is just a union of two
single valued branches.
These branches are still not everywhere defined (they are defined exactly on the domain of $f_1$)
but they are single valued!
Indeed, the square root has ramification points exactly where $f_1$ is equal to $f_1(0)$ or $\infty$.
But at all such places, namely at $0$ and at $\infty$, the function $f_1$ has simple critical points.
Thus at these places, the function $j_1$ looks like $\sqrt{u^2}$, where $u$ is some local coordinate,
and this does not have any ramifications.
It is also easy to see that $j_1$ has no nontrivial monodromy around the curves, on which it is undefined
because the square root does not have monodromy around these curves.
We choose the branch of $j_1$ that is tangent to the identity at infinity.

Now finding $f_1\circ j_1^{-1}$ is easy:
Set $w=j_1(z)$, then $w=\sqrt{f_1(z)-f_1(0)}$, and
$$
f_1\circ j_1^{-1}(w)=f_1(z)=w^2+f_1(0)
$$
for all $w$ such that the left-hand side is defined.
So this is again a quadratic polynomial!
By the way, the number $c_1=f_1(0)$ is easy to compute:
$$
f_1(0)=j(-3)=-\sqrt 6=-2^{1/2}\cdot 3^{1/2}.
$$
The only non-trivial part is the sign of the square root.
It is determined by our choice of the branch for $j$ but we skip the corresponding computation.

Next, we define the function $f_2=j_1\circ p_{c_1}$.
Note that $j_1=p_{c_1}^{-1}\circ f_1$ on the domain of $f_1$.
It follows that
$$
f_2``="p_{c_1}^{-1}\circ f_1\circ p_{c_1}.
$$
This formula looks nice but one needs to be very careful, because the expression in the right-hand side is ambiguous
(it should be considered as a single-valued branch over some domain, but then it does not carry any information on the
domain of definition).
The right understanding of this formula is that $p_{c_1}^{-1}\circ f_1$ should be thought of as $j_1$ but then
it coincides with the formula $f_2=j_1\circ p_{c_1}$ that we have used to define $f_2$.
In our formulas, one can recognize the Thurston iteration but in a slightly unusual context because
we deal with discontinuous holomorphic functions rather than with continuous non-holomorphic functions.
There is a precise relation between what is happening and Thurston's theory
(better seen on other examples, because, in the case under consideration, Thurston's theory does not have to say much).

Continuing the same process, we obtain a sequence $f_n$ of functions,
each defined on the complement to a finite union of simple curves, with the following recurrence property:
$$
f_{n+1}=j_n^{-1}\circ p_{c_n},\quad c_n=f_n(0),
$$
where $j_n$ is a branch of $\sqrt{f_n-c_n}$ defined on the domain of $f_n$ and tangent to
the identity at infinity.

In our example, we can compute the numbers $c_n$ explicitly.
Consider the sequence $f_1^{\circ m}(0)$.
This sequence stabilizes at the second term:
$$
-\sqrt 6,\ \sqrt 6,\sqrt 6,\dots,\sqrt 6,\dots
$$
The sequence $f_2^{\circ m}(0)$ can be obtained from the first sequence as follows:
$$
f_2^{\circ m}(0)=\sqrt{f_1^{\circ {m+1}}(0)-f_1(0)}.
$$
Thus all terms are equal to $\pm\sqrt{2\sqrt 6}=2^{1-1/4}\cdot 3^{1/4}$.
The determination of signs is a bit tricky, but the right signs are the following:
the first sign is minus, and all other signs are plus.
Continuing in the same way, we obtain that
$$
f_n^{\circ m}(0)=\pm 2^{1-1/2^n}\cdot 3^{1/2^n},
$$
where the first sign is minus, and all other signs are plus.
In particular, $c_n=f_n(0)\to -2$ as $n\to\infty$, and the convergence is exponentially fast.

We see that the sequence of polynomials $p_{c_n}$ converges to the polynomial $g:z\mapsto z^2-2$.
This is the so called {\em Tchebychev polynomial}.
The Julia set of this polynomial is equal to the segment $[-2,2]$.
Note that the orbit of the critical point $0$ is finite: $0\mapsto -2\mapsto 2$,
and $2$ is fixed.

Define a sequence of maps $\Phi_n=j_n\circ\dots j_1\circ j$.
The map $\Phi_n$ is defined and holomorphic on the complement to finitely many simple curves.
The main property of $\Phi_n$ is that $f_{n+1}=\Phi_n\circ f\circ\Phi_n^{-1}$ wherever the
right-hand side is defined, which follows from the definition.
Note also that $p_{c_n}=\Phi_{n-1}\circ f\circ\Phi_n^{-1}$ wherever the left-hand side is defined.
In our example, it can be shown that the sequence $\Phi_n$ converges uniformly to a map $\Phi$ defined
on the complement to countably many curves --- in our case, to
all iterated pullbacks of $[-\sqrt 3,\sqrt 3]$ under $f$.
The map $\Phi$ reglues all these horizontal segments into ``vertical curves''
so that the Julia set of $f$, which is a Cantor set, gets glued into the segment $[-2,2]$ under $\Phi$.
Passing to the limit as $n\to\infty$ in the identity $p_{c_n}=\Phi_{n-1}\circ f\circ\Phi_n^{-1}$,
we obtain that
$$
g=\Phi\circ f\circ\Phi^{-1}.
$$
Note that the right-hand side is only defined on the complement to countably many curves,
but it extends to the complex plane as a holomorphic function.

\section{Thurston's algorithm for quadratic maps}
In the example above, we constructed an explicit sequence of approximations to a topological regluing.
We will now define these approximations in a more general context.
First, we fix some conventions.
Let $F:Dom_F\to\overline\C$ and $G:Dom_G\to\overline\C$ be functions defined on certain subsets
$Dom_F$, $Dom_G$ of $\overline\C$.
These subsets are not necessarily open.
We write $F\subseteq G$ if $Dom_F\subseteq Dom_G$ and $F(z)=G(z)$ at every point $z\in Dom_F$.
The equality $F=G$ means $F\subseteq G$ and $G\subseteq F$.
For any set $A\subset\overline\C$, we write $F(A)$ for $\{F(z)\ |\ z\in Dom_F\cap A\}$
and $F^{-1}(A)=\{z\in Dom_F\ |\ F(z)\in A\}$.
Finally, define the composition $F\circ G$ as the map defined on $Dom_{F\circ G}=G^{-1}(Dom_F)$
and such that $F\circ G(z)=F(G(z))$ for all $z\in Dom_{F\circ G}$.

We will now describe a certain generalization of Thurston's algorithm to the case of partially defined ramified coverings.
Let $U$ be an open subset of the Riemann sphere, and let a function $F:U\to\overline\C$ be a ramified covering over its image of degree 2
with two critical points.
Assume that the critical orbits of $F$ are well-defined (hence, they lie in $U$).
Under this assumption, we produce another partially defined ramified covering with well-defined critical orbits.

\begin{lemma}
  For any pair of different points $a,b\in\overline\C$, there exists a quadratic rational function
  with critical values $a$ and $b$.
\end{lemma}

The statement is true and well known for any degree but the quadratic case is more explicit.

\begin{proof}
  Set
  $$
  R_{a,b}(z)=\left\{
  \begin{array}{cl}
    \frac{az^2-b}{z^2-1},& a,b\ne \infty\\
    z^2+b,& a=\infty,\ b\in\C\\
    z^{-2}+a,& a\in\C,\ b=\infty
  \end{array}\right.
  $$
  Then the critical points of $R_{a,b}$ are $\infty$ and $0$, and we have
  $$
  R_{a,b}(\infty)=a,\quad R_{a,b}(0)=b.
  $$
\end{proof}

Let $R$ be a quadratic rational function, whose critical values coincide with those of $F$.
Then the multivalued function $R^{-1}\circ F$ splits into two single valued branches.
The branches differ by post-composition with a M\"obius transformation.
However, we already have this degree of freedom in the definition of $R^{-1}\circ F$
($R$ is only defined up to pre-composition with a M\"obius transformation).
Let $J$ denote one of the branches of $R^{-1}\circ F$.
The domain of $J$ coincides with the domain of $F$ (i.e. with $U$).
Define $G=J\circ R$.
The domain of $G$ is $R^{-1}(U)$.
It is easy to see that $F\circ J^{-1}\subseteq R$ and, therefore, $J\circ F\circ J^{-1}\subseteq G$.
From this formula, it is also clear that $G$ is defined uniquely up to M\"obius conjugation.

We need to prove that the critical orbits of $G$ are well-defined.
Take any critical point of $F$.
Without loss of generality, we can assume that this critical point is $0$.
Recall that the $F$-orbit of $0$ is in the domain of $F$ and hence in the domain of $J$.
The points $J\circ F^{\circ n}(0)$ form the critical orbit of $J\circ F\circ J^{-1}$.
It is well-defined, therefore, the critical orbit of $G$ is also well defined and is the same.
The operation that takes $F$ and produces $G$ is the step in the (generalized) Thurston algorithm.
We will apply this step repeatedly in a situation described below.

\subsection{Holomorphic regluing of quadratic polynomials}
\label{s_hroqp}

Let $f$ be a quadratic polynomial.
By a suitable change of coordinates, we can write $f$ in the form $p_c(z)=z^2+c$.
Note that $c=f(0)$.
Consider a simple path $\alpha_0:[-1,1]\to\overline\C$ such that $\alpha_0(-t)=-\alpha_0(t)$ for all $t\in[-1,1]$.
The path $\alpha_0$ can be interpreted as an $\alpha$-path
$$
(t_1,t_2)\mapsto \alpha_0(t_1).
$$
We will sometimes use this interpretation when referring to topological regluing, e.g.
regluing of $\alpha_0$ means regluing of this $\alpha$-path.

If we want to do a topological regluing of the pullbacks of $\alpha_0$,
then we need to assume that $f^{\circ n}\circ\alpha_0[-1,1]$ is disjoint from $\alpha_0[-1,1]$.
(it follows that $\alpha_0[-1,1]$ is disjoint from the forward orbit of the critical value $c$).
Note that a path $\alpha_0$ from Section \ref{s_trqp} satisfies this assumption.
However, we will now work under the more general assumption that the forward orbit of $\alpha_0(1)$
is disjoint from $\alpha_0[-1,1]$.
We will define a sequence of functions holomorphic on the complements to finite unions of curves.
These function will (in many cases) approximate maps generalizing topological regluings
to the case of intersecting curves.

Consider the branch
$$
j(z)=\sqrt{z^2-\alpha_0(1)^2}
$$
that is defined on the complement to $\alpha_0[-1,1]$ and tangent to the identity at $\infty$.
The inverse function $j^{-1}$ is defined on the complement to $\beta_0[-1,1]$, for some
simple path $\beta_0$, and is given by the formula $j^{-1}(w)=\sqrt{w^2+\alpha_0(1)^2}$.
The composition $f\circ j^{-1}$ extends to the Riemann sphere as the quadratic polynomial $p_{c_0}$, where $c_0=\alpha_0(1)^2+c$:
$$
f\circ j^{-1}(z)=z^2+\alpha_0(1)^2+c=p_{c_0}(z),\quad c_0=\alpha_0(1)^2+c=f(\alpha_0(1)).
$$
Set $f_1=j\circ p_{c_0}$.
The function $f_1$ is defined on the complement to $p_{c_0}^{-1}(\alpha_0[-1,1])$.
This is a pair of disjoint simple curves since $\alpha_0[-1,1]$ does not contain $c_0=f(\alpha_0(1))$,
hence $0$ is not in $p_{c_0}^{-1}(\alpha_0[-1,1])$.
For the same reason, the number $c_1=f_1(0)$ is well defined.
From the relations $f\supseteq p_{c_0}\circ j$ and $f_1=j\circ p_{c_0}$ it follows that
$$
f_1^{\circ n}=j\circ p_{c_0}\circ\dots\circ j\circ p_{c_0}\subseteq j\circ f^{\circ\, n-1}\circ p_{c_0}.
$$
Note that the critical points of $f_1$ are $0$ and $\infty$, the infinity being fixed.
The points $f_1^{\circ n}(0)$ are well defined for all $n\ge 0$.
We can prove this by induction as follows.
For $n=1$, the statement is already proved.
Suppose now that $f_1^{\circ n}(0)$ is defined.
Then it is equal to $j\circ f^{\circ\, n-1}(c_0)=j\circ f^{\circ n}(\alpha_0(1))$.
Clearly, $p_{c_0}\circ j\circ f^{\circ n}(\alpha_0(1))$ is also defined.
Since $p_{c_0}\circ j\subseteq f$, the latter is equal to $f^{\circ\, n+1}(\alpha_0(1))$.
Finally, by our assumption on $\alpha_0$, the point $f^{\circ\, n+1}(\alpha_0(1))$ is disjoint from $\alpha_0[-1,1]$.
Therefore, $j\circ f^{\circ\, n+1}(\alpha_0(1))$ is also defined.
We can conclude that
$$
f_1^{\circ\, n+1}(0)=j\circ p_{c_0}\circ f_1^{\circ n}(0)
$$
is defined and equal to $j\circ f^{\circ\, n+1}(\alpha_0(1))$.

We can now apply Thurston's algorithm to $f_1$.
Using Thurston's algorithm, we obtain a sequence of functions $f_n$, $j_n$ and $R_n$ for $n=1,2,3,\dots$.
The function $R_n$ is a quadratic rational function, whose critical values coincide with those of $f_n$.
We can take $R_n=p_{c_n}$, where $c_n=f_n(0)$.
The function $j_n$ is a branch of $R_n^{-1}\circ f_n=\sqrt{f_n-c_n}$.
In the polynomial situation, we can fix the branch by imposing the following property:
$j_n$ is tangent to the identity at infinity.
Finally, the next function $f_{n+1}$ is defined as $j_n\circ R_n$.
Suppose that Thurston's algorithm converges, i.e. the sequence $c_n$ converges to a complex number $c_\infty$.
We say in this case that the quadratic polynomial $p_{c_\infty}$ is obtained from $f$
by {\em weak holomorphic regluing} along $\alpha_0$.

Define a sequence of maps $\Phi_n=j_n\circ\dots j_1\circ j$.
The map $\Phi_n$ is defined and continuous on the complement to finitely many simple paths.
The main property of $\Phi_n$ is that $f_{n+1}\supseteq\Phi_n\circ f\circ\Phi_n^{-1}$,
which follows from the relation $f_{n+1}=j_n\circ R_n\supseteq j_n\circ f_n\circ j_n^{-1}$.
Note, however, that $f_{n+1}$ is defined on a larger set than $\Phi_n\circ f\circ\Phi_n^{-1}$.
There is a countable union $Z$ of simple curves such that all $\Phi_n$ are defined on the complement to $Z$.
The complement to $Z$ is, in general, neither open nor connected.
However, it is given as the intersection of countably many open dense sets.
Therefore, $\overline\C-Z$ is dense.
Suppose that the sequence $\Phi_n$ converges to some function $\Phi$ uniformly on $\overline\C-Z$.
In this case, we say that $p_{c_\infty}$ is obtained from $f$ by
{\em  strong holomorphic regluing} along $\alpha_0$.
The limit $\Phi$ may or may not be a topological regluing.
For one thing, it may fail to be injective.
Also, the domain of $\Phi$ is the complement to countably many simple curves but not necessarily disjoint.
In particular, the domain may be disconnected.
Thus, strong holomorphic regluing may be used to define more general surgeries than topological regluing.

\subsection{Holomorphic regluing of quadratic rational functions}
\label{s_hroqrf}
Consider a quadratic rational function $f$ that is not a polynomial and not conjugate to $z\mapsto 1/z^2$.
Then we can reduce $f$ to the form
$$
R_{a,b}(z)=\frac{az^2-b}{z^2-1}
$$
Thus, in this section, we assume that $f=R_{a,b}$.
Note that $a=f(\infty)$ and $b=f(0)$.

Consider a simple path $\alpha_0:[-1,1]\to\overline\C$ such that $\alpha_0(-t)=-\alpha_0(t)$ for all $t\in[-1,1]$.
We will assume that the orbits of $1$ and $\alpha_0(1)$ under $f$ are disjoint from $\alpha_0[-1,1]$.
Note that $f(1)=\infty$ so that the orbit of $1$ contains the orbit of the critical point $\infty$.
Consider the branch
$$
j(z)=\sqrt{\frac{z^2-\alpha_0(1)^2}{1-\alpha_0(1)^2}}
$$
defined over the complement to $\alpha_0[-1,1]$ and such that $j(1)=1$.
The composition $f\circ j^{-1}$ extends to the Riemann sphere as a quadratic rational function $R$.
An easy computation shows that $R(0)=f(\alpha_0(1))$ and $R(\infty)=f(\infty)$.
Set $f_1=j\circ R$.
This function is defined on the complement to a pair of disjoint simple curves.
Exactly as in the case of quadratic polynomials, we use the formulas $f\supseteq R\circ j$ and $f_1=j\circ R$
to prove the formula
$$
f_1^{\circ n}\subseteq j\circ f^{\circ\, n-1}\circ R.
$$
The critical orbits of $f_1$ are well defined.
Indeed, $f_1^{\circ n}(0)=j\circ f^{\circ n}(\alpha_0(1))$, which can be proved exactly as for polynomials,
$f_1^{\circ n}(\infty)=j\circ f^{\circ n}(\infty)$, and $j$ is defined on the orbits of $\alpha_0(1)$ and $\infty$.
Note also that $f_1(1)=\infty$.

Applying Thurston's algorithm to $f_1$, we obtain a sequence of functions $f_n$, $j_n$ and $R_n$.
The function $f_n$ is defined and holomorphic on the complement to finitely many simple curves,
its critical points are $0$ and $\infty$, and we have $f_n(1)=\infty$.
The function $R_n$ is a rational function, whose critical values coincide with those of $f_n$.
We choose $R_n$ to have the form $R_{a_n,b_n}$ for some complex numbers $a_n$ and $b_n$.
The multivalued function $R_n^{-1}\circ f_n$ splits into two single valued branches on the domain of $f_n$.
Since $R_n(1)=\infty$ and $f_n(1)=\infty$, we can choose a branch $j_n$ with the property $j_n(1)=1$
It is clear that $j_n(0)=0$ and $j_n(\infty)=\infty$.
Finally, the next function $f_{n+1}$ is defined inductively as $j_n\circ R_n$.
Its critical points are $0$ and $\infty$, and we have $f_{n+1}(1)=j_n(\infty)=\infty$.

Suppose that both sequences $a_n$ and $b_n$ converge to some complex numbers $a_\infty$ and $b_\infty$.
Then we say that the rational function $R_{a_\infty,b_\infty}$ is obtained from $f$ by
{\em weak holomorphic regluing} along $\alpha_0$.
We can also define strong holomorphic regluing of quadratic rational functions
in the same way as for quadratic polynomials.

\subsection{Some questions}
The main question is: what conditions guarantee the existence of weak/strong holomorphic regluing?
In the case where $f_1$ is critically finite, this question is closely related to Thurston's theory.
Suppose that strong holomorphic regluing exists.
What can be said about the limit map $\Phi$?
To what extent is this map holomorphic?
Note that it is holomorphic in the interior of its domain but there must be an appropriate notion
generalizing holomorphicity to the boundary points (an attempt to define this notion is made in
preprint \cite{Timorin_holreg}).

\section{Existence of topological regluing}
\label{s_moore}

{\footnotesize
In this section, we prove that a contracted set of disjoint $\alpha$-paths can be reglued (Theorem \ref{conj_tr}).
We also establish necessary conditions for maps obtained by regluings to be ramified coverings (Theorem \ref{reg_rc}).
Our methods are based on a theory of Moore \cite{Moore_foundations},
who gave a purely topological characterization of topological spheres.
}

\subsection{A variant of Moore's theory}
\label{ss_moore}
Moore \cite{Moore_foundations} defined a system of topological conditions that are necessary and sufficient
for a topological space to be homeomorphic to the sphere.
He used this system to lay axiomatic foundations of plane topology.
One of the most remarkable applications of Moore's theory is a description
of equivalence relations on the sphere such that the quotient space is homeomorphic
to the sphere.

In this section, we will give a variant of Moore's theory.
Our axiomatics is completely different.
It does not have the purpose of giving a foundation to plane topology, but
just serves as a fast working tool to prove that something is a topological sphere.
The main idea, however, remains the same.

Let $X$ be a compact connected Hausdorff second countable topological space.
Recall that a {\em simple closed curve} in $X$ is the image in $X$ of a continuous
embedding $\gamma:S^1\to X$.
Here the map $\gamma$ is called a {\em simple closed path}.
We also define a {\em simple path} as a continuous embedding of $[0,1]$ into $X$,
and a {\em simple curve} as the image of a simple path.
A {\em segment} of a simple curve is defined as the image of a subsegment in $[0,1]$
under the corresponding simple path.
Similarly, we can define a segment of a simple closed curve.

Suppose we fixed some set $\Ec$ of simple curves in $X$.
The curves in $\Ec$ are called {\em elementary curves}.
We will always assume that segments of elementary curves are also elementary curves,
and that if two elementary curves have only an endpoint in common, then
their union is also an elementary curve.
We will not state these assumptions as axioms, although, technically, they are.
Define an {\em elementary closed curve} as a simple closed curve, all of
whose segments are elementary curves.

We are ready to state the first axiom:

\begin{ax}[Jordan domain axiom]
  Any elementary closed curve divides $X$ into two connected components called
  {\em elementary} domains.
\end{ax}

Since a simple closed curve is homeomorphic to a circle, it makes sense to talk
about the (circular) order of points on it.
A {\em topological quadrilateral} in $X$ is defined as a domain bounded by a
simple closed curve $J\subset X$ with a distinguished quadruple of different points $a$, $b$, $c$ and $d$ on $J$
positioned on the curve in this order.
The domain is called the {\em interior} of the quadrilateral.
An {\em elementary quadrilateral} is a topological quadrilateral, whose bounding curve is elementary.
A simple curve connecting the segment $[a,b]$ with the segment $[c,d]$ of $J$ is called a {\em vertical curve},
provided that the interior of it lies in the interior of the quadrilateral.
A simple curve connecting the segments $[b,c]$ and $[d,a]$ is called {\em horizontal}, provided that
its interior is in the interior of the quadrilateral.
We will also regard $[a,b]$ and $[c,d]$ as horizontal segments,
$[b,c]$ and $[d,a]$ as vertical segments.

\begin{ax}[Extension axiom]
  For any elementary quadrilateral, and a point $x$ on the boundary of it but not in a vertex,
  there exists a vertical or a horizontal elementary curve with an endpoint at $x$.
\end{ax}

Define a {\em grid} in an elementary quadrilateral $Q$ as a system of
finitely many horizontal elementary curves and
finitely many vertical elementary curves such that all horizontal curves are pair-wise disjoint and all vertical curves are
pair-wise disjoint, and every horizontal curve intersects every vertical curve at exactly one point.
Using the Jordan domain axiom, it is easy to see that every grid with $n-1$ horizontal and
$m-1$ vertical curves divides $Q$ into $mn$ pieces.
We will refer to these pieces as {\em cells of the grid}.
Cells can be regarded as elementary quadrilaterals.

\begin{ax}[Covering axiom]
  Consider an elementary quadrilateral $Q$ and an open covering $\Uc$ of $\overline Q$.
  Then there exists a grid in $ Q$ such that the closure of every cell belongs to
  a single element of $\Uc$ (such grid is said to be {\em subordinate} to $\Uc$).
\end{ax}

The following is the main fact from Moore's theory that we need.
This particular formulation is new, but it is inspired by \cite{Moore_foundations}.
As above, $X$ is a compact connected Hausdorff second countable topological space.

\begin{theorem}
\label{moore}
  Suppose that a set of elementary curves $\Ec$ in $X$ satisfies the Jordan domain axiom, the Extension axiom and the Covering axiom.
  Also suppose that there exists an elementary closed curve in $X$.
  Then $X$ is homeomorphic to the sphere.
\end{theorem}

\begin{proof}
Since there exists an elementary closed curve, by the Jordan domain axiom, there exists
an elementary quadrilateral $ Q$.
It suffices to prove that the closure of this quadrilateral is homeomorphic to the closed disk.

Fix a countable basis $\Bc$ of the topology in $X$.
There are countably many finite open coverings of $ Q$ contained in $\Bc$.
Number all such coverings by natural numbers.
We will define a sequence of grids $G_n$ in $ Q$ by induction on $n$.
For $n=1$, we just take the trivial grid, the one that does not have any horizontal or vertical curves.
Suppose now that $G_n$ is defined.
Let $\Uc_n$ be the $n$-th covering of $ Q$.
Using the Covering axiom, we can find a grid in each cell of $G_n$ that is subordinate to $\Uc_n$.
Using the Extension axiom, we can extend these grids to a single grid $G_{n+1}$ in $ Q$.
Thus $G_{n+1}$ contains $G_n$ and is subordinate to $\Uc_n$.

Consider any pair of different points $x,y\in Q$.
There exists $n$ such that $x$ and $y$ do not belong to the closure of the same cell in $G_{n+1}$.
Indeed, let us first define a covering of $ Q$ as follows.
For any $z\in\overline Q$, choose $U_z$ to be an element of $\Bc$ that contains $z$ but does
not include the set $\{x,y\}$.
The sets $U_z$ form an open covering of $\overline Q$.
Since $\overline Q$ is compact, there is a finite subcovering.
This finite subcovering coincides with $\Uc_n$ for some $n$.
Then, by our construction, the closure of every cell in $G_{n+1}$ is contained in a single
element of $\Uc_n$.
However, the set $\{x,y\}$ is not contained in a single element of $\Uc_n$.
Therefore, $\{x,y\}$ cannot belong to the closure of a single cell.

Consider a nested sequence $C_1\subset C_2\supset\dots\supset C_n\supset\dots$, where
$C_n$ is a cell of $G_n$.
We claim that the intersection of the closures $\overline C_n$ is a single point.
Indeed, this intersection is nonempty, since $\overline C_n$ form a nested sequence of
nonempty compact sets.
On the other hand, as we saw, there is no pair $\{x,y\}$ of different points
contained in all $\overline C_n$.

Consider the standard square $[0,1]\times [0,1]$ and a sequence $H_n$ of
grids in it with the following properties:
\begin{enumerate}
  \item all horizontal curves in $H_n$ are horizontal straight segments, all
  vertical curves in $H_n$ are vertical straight segments;
  \item the grid $H_n$ has the same number of horizontal curves and the same
  number of vertical curves as $G_n$, thus there is a natural one-to-one
  correspondence between cells of $H_n$ and cells of $G_n$ respecting ``combinatorics'',
  i.e. the cells in $H_n$ corresponding to adjacent cells in $G_n$ are also adjacent;
  \item the grid $H_{n+1}$ contains $H_n$; moreover, if a cell of $H_{n+1}$ is in
  a cell of $H_n$, then there is a similar inclusion between the corresponding cells of
  $G_{n+1}$ and $G_n$;
  \item between any horizontal segment of $H_n$ and the next horizontal segment, the horizontal segments of $H_{n+1}$
  are equally spaced; similarly, between any vertical segment of $H_n$ and the next vertical segment, the vertical
  segments of $H_{n+1}$ are equally spaced.
\end{enumerate}
It is not hard to see that any nested sequence of cells $D_n$ of $H_n$ converges to a point:
$\bigcap\overline D_n=\{pt\}$.

We can now define a map $\Phi:\overline Q\to [0,1]\times [0,1]$ as follows.
For a point $x\in\overline Q$, there is a nested sequence of cells $C_n$
of $G_n$ such that $\overline C_n$ contains $x$ for all $n$.
Define the point $\Phi(x)$ as the intersection of the closures of the
corresponding cells $D_n$ in $H_n$.
(Note that they also form a nested sequence according to our assumptions).
Clearly, the point $\Phi(x)$ does not depend on a particular choice of the
nested sequence $C_n$ (there can be at most four different choices).
It is also easy to see that $\Phi$ is a homeomorphism between $\overline Q$ and
the standard square $[0,1]\times [0,1]$.
\end{proof}

One of the main applications of Moore's theory is the following characterization
of equivalence relations on $S^2$ with quotients homeomorphic to the sphere:

\begin{theorem}
  \label{equiv}
  Let $\sim$ be an equivalence relation on $S^2$ such that all equivalence classes
  are compact and connected, they do not separate the sphere, and form an upper semicontinuous family.
  Then the quotient $S^2/\sim$ is homeomorphic to the sphere provided that not all
  points are equivalent.
\end{theorem}

Recall that a family of compact sets $\Sc$ is {\em upper semicontinuous} if for any $A\in\Sc$
and any neighborhood $U$ of $A$, there is another neighborhood $V$ of $A$ such that
$B\subset U$ whenever $B\in\Sc$ and $B\cap V\ne\emptyset$.

\subsection{Regluing space}
\label{s_regs}
Let $\Ac$ be a contracted set of disjoint $\alpha$-paths.
In this section, we prove that there is a regluing of $\Ac$ into some contracted set $\Bc$ of disjoint $\beta$-paths
(Theorem \ref{conj_tr}).
To this end, we first define an abstract topological space $X$, a set $\Bc$
of disjoint $\beta$-paths in $X$, and a homeomorphism $\Phi:S^2-\Im\Ac\to X-\Im\Bc$.
Then we show that $X$ is homeomorphic to the sphere and that $\Phi$ is a regluing.

Let $\Bc$ be an abstract set, whose elements are in one-to-one correspondence with the paths in $\Ac$.
We will write $\beta_\alpha$ for the path in $\Bc$ corresponding to a path $\alpha\in\Ac$.
As a set of points, $X$ is defined to be the disjoint union of $X_0=S^2-\Im\Ac$
and the set of points of the form $\beta(t_1,t_2)$, where $\beta\in\Bc$ and $(t_1,t_2)\in S^1$.
Here $\beta(t_1,t_2)$ is thought of as a formal expression defined by $\beta$ and $(t_1,t_2)$.
We identify $\beta(t_1,t_2)$ with $\beta(-t_1,t_2)$ but make no further identifications.

Let $\Ac'$ be a finite subset of $\Ac$.
Clearly, there exists a regluing $\Phi_{\Ac'}$ of $\Ac'$ into some finite set $\Bc'$ of
disjoint simple paths.
Define the following equivalence relation $\sim_{\Ac'}$ on $S^2$: two different points
are equivalent if and only if they belong to the same component of $\Phi_{\Ac'}(\Im(\Ac-\Ac'))$.
(Note that $\Phi_{\Ac'}$ is defined on $\Im(\Ac-\Ac')$ because $\Im(\Ac-\Ac')$ is disjoint
from $\Im\Ac'$).
By Theorem \ref{equiv}, the quotient $Y_{\Ac'}=S^2/\sim_{\Ac'}$ is homeomorphic to the sphere.
Note that we need to use our condition on $\Ac$ to guarantee that $\sim_{\Ac'}$ satisfies
all assumptions of Theorem \ref{equiv} including the upper semicontinuity.
Let $\pi_{\Ac'}:S^2\to Y_{\Ac'}$ denote the canonical projection.
Note that $Z_{\Ac'}=\pi_{\Ac'}(\Phi_{\Ac'}(\Im(\Ac-\Ac')))$ is a countable subset of $Y_{\Ac'}$.
A standard Baire category argument shows that paths in $Y_{\Ac'}$ avoiding $Z_{\Ac'}$ are
dense in the space of paths with the uniform metric.

We can now define a map $\phi_{\Ac'}:X\to Y_{\Ac'}$ as follows.
If $x\in X_0$, then we set $\phi_{\Ac'}(x)=\pi_{\Ac'}\circ\Phi_{\Ac'}(x)$.
If $\alpha\in\Ac'$, then we define $\phi_{\Ac'}(\beta_\alpha(t_1,t_2))$ to be $\pi_{\Ac'}(\beta'(t_1,t_2))$,
where $\beta'$ is the $\beta$-path in $\Bc'$ corresponding to the $\alpha$-path $\alpha$
(it is important not to confuse $\beta_\alpha$, which is a formal object, with $\beta'$, which
is a $\beta$-path in $\Bc'$).
Finally, if $\alpha\in\Ac-\Ac'$, then we define $\phi_{\Ac'}(\beta_\alpha(t_1,t_2))$ to be
$\pi_{\Ac'}\circ\Phi_{\Ac'}(\alpha(S^1))$, which is a single point.
Choose a finite subset $\Ac'$ of $\Ac$ and a simple path $\gamma:[0,1]\to Y_{\Ac'}$
avoiding the set $Z_{\Ac'}$.
Then $\phi_{\Ac'}^{-1}\circ\gamma$ is a well-defined map from $[0,1]$ to $X$.
We call this map an {\em elementary path} in $X$.
Consider now a simple closed path $\gamma:S^1\to Y_{\Ac'}$ such that $\gamma(S^1)$
avoids the set $Z_{\Ac'}$.
By the Jordan theorem, the image of $\gamma$ divides $Y_{\Ac'}$ into two domains $U$ and $U'$.
The preimages of $U$ and $U'$ under $\phi_{\Ac'}$ are called {\em elementary domains in $X$}.

\begin{proposition}
  Elementary domains in $X$ form a basis of some Hausdorff topology.
\end{proposition}

\begin{proof}
  Let $\Ac_1$ and $\Ac_2$ be two finite subsets of $\Ac$.
  Consider the corresponding spaces $Y_{\Ac_1}$ and $Y_{\Ac_2}$.
  Fix Jordan domains $U_1$ in $Y_{\Ac_1}$ and $U_2$ in $Y_{\Ac_2}$, whose
  boundaries avoid $Z_{\Ac_1}$ and $Z_{\Ac_2}$, respectively.
  Set $V_1=\phi_{\Ac_1}^{-1}(U_1)$ and $V_2=\phi_{\Ac_2}^{-1}(U_2)$.
  We need to prove that, for any $x\in V_1\cap V_2$, there exists an
  elementary domain $V_3\ni x$ such that $V_3\subseteq V_1\cap V_2$.
  Set $\Ac_3=\Ac_1\cup\Ac_2$.
  There are continuous maps $\psi_{\Ac_1,\Ac_3}:Y_{\Ac_3}\to Y_{\Ac_1}$
  and $\psi_{\Ac_2,\Ac_3}:Y_{\Ac_3}\to Y_{\Ac_2}$ such that
  $$
  \phi_{\Ac_1}=\psi_{\Ac_1,\Ac_3}\circ\phi_{\Ac_3},\quad
  \phi_{\Ac_2}=\psi_{\Ac_2,\Ac_3}\circ\phi_{\Ac_3}.
  $$
  The sets $\tilde U_1=\psi_{\Ac_1,\Ac_3}^{-1}(U_1)$ and
  $\tilde U_2=\psi_{\Ac_2,\Ac_3}^{-1}(U_2)$ are open in $Y_{\Ac_3}$.
  Moreover, we have $\phi_{\Ac_3}(x)\in\tilde U_1\cap\tilde U_2$.
  There is a Jordan domain $U_3$ in $\tilde U_1\cap\tilde U_2$ containing
  $\phi_{\Ac_3}(x)$, whose boundary is a simple closed curve avoiding $Z_{\Ac_3}$.
  We set $V_3=\phi_{\Ac_3}^{-1}(U_3)$.

  To prove that the topology defined by elementary domains is Hausdorff, we
  consider any pair of different points $x,y\in X$ and the set $\Ac'$ of all
  $\alpha\in\Ac$ such that $x,y\in\beta_\alpha(S^1)$ (note that the set $\Ac'$
  is either empty or contains only one path).
  Clearly, the points $\phi_{\Ac'}(x)$ and $\phi_{\Ac'}(y)$ of $Y_{\Ac'}$
  belong to disjoint Jordan domains, whose boundaries are also disjoint and
  avoid the set $Z_{\Ac'}$.
  The pullbacks of these domains under $\phi_{\Ac'}$ separate the points $x$ and $y$.
\end{proof}

\begin{proposition}
\label{cont}
  For any finite subset $\Ac'\subset\Ac$, the map $\phi_{\Ac'}:X\to Y_{\Ac'}$ is continuous.
\end{proposition}

\begin{proof}
  Given a point $x\in X$ and an open neighborhood $U$ of $\phi_{\Ac'}(x)$, we need to find an open
  neighborhood $V$ of $x$ such that $\phi_{\Ac'}(V)\subseteq U$.
  There is a Jordan curve around $\phi_{\Ac'}(x)$ in $Y_{\Ac'}$ that lies in $U$ and avoids the set $Z_{\Ac'}$.
  Let $U'$ be the Jordan domain bounded by this curve and contained in $U$.
  We can set $V=\phi_{\Ac'}^{-1}(U')$.
\end{proof}

It is clear that $X$ is second countable.

\begin{proposition}
\label{compact}
  The space $X$ is compact.
\end{proposition}

\begin{proof}
  We will prove that any sequence in $X$ has a convergent subsequence.
  Consider a sequence $x_n\in X$.
  The sequence $\phi_{\emptyset}(x_n)\in Y_\emptyset$ has a convergent subsequence.
  We may assume without loss of generality that $\phi_\emptyset(x_n)\to y$.
  Let $\Ac'$ be the set of all $\alpha\in\Ac$ such that $\phi_\emptyset^{-1}(y)\subseteq\beta_\alpha(S^1)$.
  Note that $\Ac'$ contains at most one element.
  We may assume that $y'_n=\phi_{\Ac'}(x_n)$ also converges to some point $y'\in Y_{\Ac'}$.
  Then $x_n$ converges to $x=\phi_{\Ac'}^{-1}(y')$.
  Indeed, let $V\ni x$ be any open neighborhood.
  We may assume without loss of generality that $V=\phi_{\Ac''}^{-1}(U)$, where
  $\Ac''\supseteq\Ac'$, and $U$ is bounded by a Jordan curve avoiding $Z_{\Ac''}$.
  It is easy to see that the sequence $y''_n=\phi_{\Ac''}(x_n)$ converges to
  $y''=\psi_{\Ac',\Ac''}^{-1}(y')$.
  Therefore, all $y''_n$ but finitely many belong to $U$.
  It follows that all $x_n$ but finitely many belong to $V$.
\end{proof}

We will need the following general fact:

\begin{proposition}
\label{inv_neigh}
  Let $Y$ be a metric space, and $\phi:X\to Y$ a continuous map.
  For any compact subset $C\subset Y$ and an open neighborhood $V$ of $\phi^{-1}(C)$,
  there exists an open neighborhood $U$ of $C$ such that $\phi^{-1}(U)\subseteq V$.
\end{proposition}

\begin{proof}
  Consider a nested sequence of closed sets $C_n$ such that $C$ is contained in
  the interior of every $C_n$, and the intersection of all $C_n$ is $C$.
  Set $D_n=\phi^{-1}(C_n)$.
  These are closed sets containing $D=\phi^{-1}(C)$ in their interiors, and the intersection of all these sets
  is equal to $D$.
  Suppose that all $D_n-V$ are nonempty.
  Then the intersection of these sets is also nonempty, by compactness of $X$.
  This is a contradiction, which shows that $D_n\subset V$ for some $n$.
  It suffices to set $U$ to be the interior of $C_n$.

\end{proof}

\begin{corollary}
  Any elementary path in $X$ is continuous.
\end{corollary}

\begin{proof}
  Let $\Ac'\subset\Ac$ be a finite subset, and $\gamma:S^1\to Y_{\Ac'}$ a closed simple path
  avoiding the set $Z_{\Ac'}$.
  We need to show that $\phi_{\Ac'}^{-1}\circ\gamma$ is continuous.
  Indeed, by Proposition \ref{inv_neigh}, the map $\phi_{\Ac'}^{-1}$ is continuous on $\gamma(S^1)$.

\end{proof}

\begin{theorem}
  The space $X$ together with the set of elementary paths in $X$ satisfies the Jordan domain axiom,
  the Extension axiom, and the Covering axiom.
  Therefore, $X$ is homeomorphic to the sphere.
\end{theorem}

\begin{proof}
  To prove the Jordan domain axiom, we just need to verify that any elementary domain is connected.
  Consider an elementary domain $V\subset X$.
  There exists a finite subset $\Ac'\subset\Ac$ and a Jordan domain $U$ in $Y_{\Ac'}$,
  whose boundary avoids $Z_{\Ac'}$, such that $V=\phi_{\Ac'}^{-1}(U)$.
  We will prove that $V$ is path connected.
  Choose any pair of points $x,y\in V$.
  Let $\Ac''$ be a finite subset of $\Ac$ containing $\Ac'$ and all paths $\alpha\in\Ac$
  such that $\beta_\alpha(S^1)$ intersects $\{x,y\}$ (there are at most two such paths).
  The images $\phi_{\Ac''}(x)$ and $\phi_{\Ac''}(y)$ can be connected in $Y_{\Ac''}$ by
  a path avoiding the set $Z_{\Ac''}$.
  The lift of this path to $X$ is an elementary path connecting $x$ with $y$.

  Let $ Q$ be an elementary quadrilateral in $X$, and $x$ a point on the
  boundary of $ Q$ but not at a vertex.
  Consider a finite subset $\Ac'$ of $\Ac$ such that $\phi_{\Ac'}( Q)$
  is bounded by a simple closed path in $Y_{\Ac'}$ avoiding $Z_{\Ac'}$.
  There is a horizontal or vertical curve with respect to the quadrilateral $\phi_{\Ac'}( Q)$
  having $\phi_{\Ac'}(x)$ as an endpoint and avoiding the set $Z_{\Ac'}$.
  The lift of this curve to $X$ is a horizontal or vertical curve with respect to $ Q$
  that passes through $x$.
  Thus the extension axiom holds.

  The covering axiom can be proved in the same way.
\end{proof}

\begin{proof}[Proof of Theorem \ref{conj_tr}]
We just need to define a regluing $\Phi:X_0\to X-\Bc$ (here $X_0$ is considered as a subset of $S^2$).
Set $\Phi(x)=x$ for all $x\in X_0$ (recall that, by definition of $X$, the set $X_0$ is a part of $X$).
When talking about $X_0$ (which sits in both the source and the target spaces), we will
always mean the topology induced from $S^2$ (the source space).
Then $\Phi(X_0)$ is $X_0$ with the topology induced from $X$.

First, we need to show that $\Phi:X_0\to\Phi(X_0)$ is continuous.
In other terms, $U\cap X_0$ is open in $X_0$ for any elementary domain $U$ in $X$.
Indeed, $U=\phi_{\Ac'}^{-1}(V)$ for some finite $\Ac'\subset\Ac$ and open $V\subset Y_{\Ac'}$.
Now the statement follows from the fact that the restriction of $\phi_{\Ac'}$ to $X_0$
(which coincides with $\pi_{\Ac'}\circ\Phi_{\Ac'}$) is continuous.

Finally, consider any $\alpha$-path $\alpha\in\Ac$.
The map $\phi_{\{\alpha\}}\circ\Phi\circ\alpha$ extends to the unit circle as a $\beta$-path.
It follows that so does $\Phi\circ\alpha$.
\end{proof}

\subsection{Ramified coverings}
\label{s_ramcov}
Let $X$ be a topological sphere and $f:X\to X$ a continuous map.
Recall that $f$ is called a {\em ramified covering} at $x\in X$, if there
exist neighborhoods $U$ of $x$ and $V$ of $f(x)$ and homeomorphisms
$\phi:U\to\Delta$ and $\psi:V\to\Delta$ such that $\psi\circ f\circ\phi^{-1}$
coincides with $z\mapsto z^k$ on $\Delta$, where $k$ is some positive integer.
The number $k$ is called the {\em local degree} of $f$ at $x$.
The map $f$ is a ramified covering if it is a ramified covering locally at every point.
The following is a topological criterion for $f$ to be a ramified covering.

\begin{theorem}
\label{ramcov}
  Suppose that $y=f(x)$, and there exist simply connected domains $U\ni x$
  and $V\ni y$ such that $f:U-\{x\}\to V-\{y\}$ is a covering of degree $k$.
  Then $f$ is a ramified covering of degree $k$ at $x$.
\end{theorem}

\begin{proof}
  Both $U-\{x\}$ and $V-\{y\}$ can be homeomorphically identified with the
  quotient of the upper half-plane $\H=\{\Im z>0\}$ by the translation $z\mapsto z+1$.
  Moreover, one can choose the homeomorphisms with the following property: as a point in $\H$ tends to infinity,
  the corresponding point in $U-\{x\}$ (respectively, in $V-\{y\}$) tends to $x$ (respectively, to $y$).
  The map $f$ restricted to $U-\{x\}$ lifts to a homeomorphism $F:\H\to\H$ such that $F(z+1)=F(z)+k$.
  Consider the maps $\Phi(z)=\exp(2\pi iF(z)/k)$ and $\Psi(z)=\exp(2\pi iz)$.
  These maps descend to homeomorphisms $\phi:U-\{x\}\to\Delta-\{0\}$ and
  $\psi:V-\{y\}\to\Delta-\{0\}$, respectively.
  Moreover, we have $\psi\circ f=\phi^k$.
  Clearly, $\phi$ and $\psi$ extend continuously to $x$ and $y$, respectively, and
  $\phi(x)=\psi(y)=0$ so that the equality $\psi\circ f=\phi^k$ still holds.
\end{proof}

Recall that Theorem \ref{reg_rc} states a simple condition that guarantees that a
topological regluing of a ramified covering is also a ramified covering.
We are now ready to prove this theorem.

\begin{proof}[Proof of Theorem \ref{reg_rc}]
  Consider a topological ramified covering $f:S^2\to S^2$ and a contracted strongly $f$-stable
  set $\Ac$ of disjoint $\alpha$-paths.
  Let $\Phi:S^2-\Im\Ac\to S^2-\Im\Bc$ be a regluing of $\Ac$ into a contracted set $\Bc$ of $\beta$-paths.
  The map $g=\Phi\circ f\circ\Phi^{-1}$ extends to a continuous map from $S^2$ to $S^2$.
  When talking about $g$, we mean this continuous map.
  We want to prove that actually, $g$ is a ramified covering.

  Define $C_g$ as the set consisting of points $\Phi(c)$, where $c\not\in\Im\Ac$ is a critical point of $f$,
  and points $\beta(1,0)$, where $\beta\in\Bc$ is the $\beta$-path corresponding to some $\alpha$-path
  $\alpha\in\Ac$, and $\alpha(0,1)$ is a critical point of $f$.
  Clearly, $C_g$ is finite (and contains as many points as there are critical points of $f$).
  By Proposition \ref{ramcov}, it suffices to prove that for every $y\not\in C_g$ there exists a neighborhood
  of $y$, on which $g$ is injective.

  First, assume that $y\not\in\Im\Bc$.
  Then $y=\Phi(x)$ for some point $x\not\in\Im\Ac$.
  Consider a small neighborhood $V$ of $f(x)$ that does not contain critical values of $f$.
  Let $U$ be a Jordan neighborhood of $x$, whose boundary is disjoint from $\Im\Ac$ and which is contained in $f^{-1}(V)$.
  Clearly, $\hat\Phi(U)$ is a Jordan neighborhood of $y$, and $g$ is injective on this neighborhood.
  Here $\hat\Phi$ is defined as in the proof of Proposition \ref{bs}.

  Next, assume that $y\in\beta(S^1)$, where $\beta\in\Bc$ is the $\beta$-path corresponding to some
  $\alpha$-path $\alpha\in\Ac$, and $\alpha(S^1)$ contains no critical points of $f$.
  Consider a small neighborhood $V$ of $f(\alpha(S^1))$ containing no critical values of $f$.
  Let $U$ be a Jordan neighborhood of $\alpha(S^1)$, whose boundary is disjoint from $\Im\Ac$ and
  which is contained in $f^{-1}(V)$.
  Clearly, $\hat\Phi(U)$ is a Jordan neighborhood of $y$, and $g$ is injective on this neighborhood.

  Finally, assume that $y=\beta(t_1,t_2)$, where $\beta\in\Bc$ is the $\beta$-path corresponding to
  some $\alpha$-path $\alpha\in\Ac$, and $\alpha(0,1)$ is a critical point of $f$.
  Then we necessarily have $t_2\ne 0$.
  Choose a small Jordan neighborhood $V$ of $f(\alpha(t_1,t_2))$ that does not intersect $\Im\Ac$
  and that intersects $f\circ\alpha(S^1)$ in at most two points.
  Then a component of $\hat\Phi(f^{-1}(V))$ that contains $y$ is a Jordan neighborhood of $y$, on
  which $g$ is injective, as can be easily seen from the definition of a regluing
  (note, however, that a component of $\hat\Phi(f^{-1}(V))$ may be ``glued by $\Phi$'' from parts of
  different components of $f^{-1}(V)$).
\end{proof}

\noindent------------------------\\
\noindent Vladlen Timorin\\
School of Engineering and Science
Mathematical Sciences\\
Jacobs University Bremen \\
Germany\\  (v.timorin@jacobs-university.de)\\

\end{document}